\documentclass[11pt,english]{amsart}
\usepackage[latin9]{inputenc}
\usepackage[a4paper]{geometry}
\usepackage{verbatim}
\usepackage{amstext}
\usepackage{amsthm}
\usepackage{amssymb}

\makeatletter
\numberwithin{equation}{section}
\numberwithin{figure}{section}
\theoremstyle{plain}
\newtheorem{thm}{\protect\theoremname}[section]
  \theoremstyle{plain}
  \newtheorem{cor}[thm]{\protect\corollaryname}
  \theoremstyle{plain}
  \newtheorem{prop}[thm]{\protect\propositionname}
  \theoremstyle{definition}
  \newtheorem{example}[thm]{\protect\examplename}
  \theoremstyle{definition}
  \newtheorem{defn}[thm]{\protect\definitionname}
  \theoremstyle{remark}
  \newtheorem{rem}[thm]{\protect\remarkname}
  \theoremstyle{plain}
  \newtheorem{lem}[thm]{\protect\lemmaname}

\usepackage[english]{babel}
\usepackage{amsfonts}
\usepackage[all]{xy}
\usepackage{tikz}
\usetikzlibrary{arrows}
\xyoption{arc}
\CompileMatrices

\title[Generating series for for $E$-polynomials of character varieties]{Generating series for the $E$-polynomials of $GL(n,{\mathbb C})$-character varieties}

\author[C. Florentino]{Carlos Florentino}

\address{Departamento de Matem\'{a}tica, Faculdade de Ci\^encias, Univ. de Lisboa, Campo Grande, Edf. C6, Lisbon, Portugal}

\email{caflorentino@ciencias.ulisboa.pt}

\author[A. Nozad]{Azizeh Nozad}

\address{School of Mathematics, Institute for Research in Fundamental Sciences (IPM), P.O.Box: 19395-5746, Tehran, Iran}

\email{anozad@ipm.ir}

\author[A. Zamora]{Alfonso Zamora}

\address{Departamento Interfacultativo de Matem\'atica Aplicada y Estad\'istica, Facultad
de Ciencias Econ\'omicas y Empresariales, Universidad CEU San Pablo, Juli\'an Romea 23, 28003 Madrid, Spain}

\email{alfonso.zamorasaiz@ceu.es}

\thanks{This work was partially supported by CAMGSD and CMAF-CIO of the University of Lisbon, the projects PTDC/MAT-PUR/30234/2017, FCT Portugal, a grant from IPM, Iran and project MTM2016-79400-P by the Spanish Ministerio de Econom\'ia y Competitividad.}


\renewcommand{\hom}{\mathrm{Hom}}
\def\quot{/\!\!/}

\newcommand{\pexp}{\operatorname{PExp}}

\newcommand{\sym}{\operatorname{Sym}}

\makeatother

\usepackage{babel}
  \providecommand{\corollaryname}{Corollary}
  \providecommand{\definitionname}{Definition}
  \providecommand{\examplename}{Example}
  \providecommand{\lemmaname}{Lemma}
  \providecommand{\propositionname}{Proposition}
  \providecommand{\remarkname}{Remark}
\providecommand{\theoremname}{Theorem}

\begin{document}
\begin{abstract}
With $G=GL(n,\mathbb{C})$, let $\mathcal{X}_{\Gamma}G$ be the $G$-character
variety of a given finitely presented group $\Gamma$, and let $\mathcal{X}_{\Gamma}^{irr}G\subset\mathcal{X}_{\Gamma}G$
be the locus of irreducible representation conjugacy classes. We provide
a concrete relation, in terms of plethystic functions, between the
generating series for $E$-polynomials of $\mathcal{X}_{\Gamma}G$
and the one for $\mathcal{X}_{\Gamma}^{irr}G$, generalizing a formula
of Mozgovoy-Reineke \cite{MR}. The proof uses a natural stratification
of $\mathcal{X}_{\Gamma}G$ coming from affine GIT, the combinatorics
of partitions, and the formula of MacDonald-Cheah for symmetric products;
we also adapt it to the so-called Cartan brane in the moduli space
of Higgs bundles. Combining our methods with arithmetic ones yields
explicit expressions for the $E$-polynomials, and Euler characteristics,
of the irreducible stratum of $GL(n,\mathbb{C})$-character varieties
of some groups $\Gamma$, including surface groups, free groups, and
torus knot groups, for low values of $n$. 
\end{abstract}

\maketitle

\section{Introduction}

Let $G$ be a complex reductive algebraic group, $\Gamma$ be a finitely
presented group, such as the fundamental group of a compact manifold
or a finite $CW$-complex, and let 
\[
\mathcal{X}_{\Gamma}G=\hom(\Gamma,G)\quot G
\]
be the \emph{$G$-character variety of $\Gamma$}: the (affine) geometric
invariant theory quotient of the algebraic variety of representations
of $\Gamma$ into $G$. When the group $\Gamma$ is the fundamental
group of a Riemann surface (or more generally, a Kähler group) these
varieties are homeomorphic to moduli spaces of $G$-Higgs bundles
via the non-abelian Hodge correspondence (see, for example \cite{Sim}),
spaces which have been studied in connection to important problems
in Mathematical-Physics in the context of mirror symmetry, and in
the quantum field theory interpretation of the geometric Langlands
correspondence \cite{KW}.

The study of geometric and topological properties of character varieties
is an active topic and there are many recent advances in the computation
of their Poincaré polynomials and other invariants, especially in
the surface group case and for related groups $\Gamma$. With the
introduction of arithmetic methods, Hausel and Rodriguez-Villegas
\cite{HRV1} showed that many of these varieties are of polynomial
type, which allows, upon applying a theorem of N. Katz \cite[Appendix]{HRV1}
to infer their $E$-polynomials by counting the number of points over
finite fields. The fact that moduli spaces of Higgs bundles have pure
cohomology allows the derivation of the Poincaré polynomial from the
$E$-polynomial, and this approach was particularly successful in
the case of smooth moduli spaces (see the works of Schiffmann, Mellit
\cite{Sc,Me}, and references therein).

However, explicitly computable formulae for these polynomials are
very hard to obtain, in particular for many well known \emph{singular}
character varieties, as one can infer from the geometric methods of
Logares, Muñoz, Newstead and Lawton \cite{LMN}, \cite{LM} and from
the arithmetic approach of Baraglia and Hekmati \cite{BH}, which
become intractable for higher dimensional groups $G$.

In this article, we introduce another point of view in the computations
of $E$-polynomials of $GL(n,\mathbb{C})$-character varieties for
\emph{arbitrary} finitely presented $\Gamma$. In particular, our
methods yield formulae for $E$-polynomials of character varieties
which are not necessarily of polynomial type. The new approach is
based on a stratification of $GL(n,\mathbb{C})$-character varieties
by partition type, and relates well with geometric and arithmetic
techniques, relying also on the combinatorics of the plethystic functions,
that have been previously used with success in connection with counting
formulae for moduli spaces of polynomial type over finite fields.

This new perspective on $E$-polynomial calculations for character
varieties, unveils another connection between the representation theory
of $GL(n,\mathbb{C})$, and that of the symmetric group $S_{n}$.
A similar approach may be possible for other reductive groups $G$,
yelding a relation between effective $E$-polynomial computations
for $G$-character varieties of an arbitrary $\Gamma$, and the representation
theory of the Weyl group of $G$. Our approach is also intimately
related to the plethystic program for counting gauge invariant operators
in supersymmetric quantum field theories, where a fundamental role
is played by symmetric products of the moduli spaces of vacua (see
\cite{FHH}). In another direction, by combining our approach with
previous results on character varieties of free groups, we were able
to prove (see \cite{FNZ}) that the $E$-polynomials of $\mathcal{X}_{\Gamma}SL(n,\mathbb{C})$
and of $\mathcal{X}_{\Gamma}PGL(n,\mathbb{C})$ agree for all $n\in\mathbb{N}$,
when $\Gamma$ is a free group, an equality predicted in \cite[Rmk. 9]{LM}
(and proved there for $n=2,3$).

We now outline the article, and some of the main results. In sections
2 and 3 we present the main properties of $E$-polynomials defined
from mixed Hodge structures on complex quasi-projective varieties,
and we describe natural methods for stratifying general $G$-character
varieties in the context of affine geometric invariant theory (GIT).
Since we always work over $\mathbb{C}$, we will abbreviate $GL(n,\mathbb{C})$
to $GL_{n}$. Let $E(X;u,v)$ denote the $E$-polynomial (in two variables
$u,v$) of a quasi-projective complex variety $X$. In section 4 we
introduce the stratification by partition type of the character varieties
$\mathcal{X}_{\Gamma}GL_{n}$, for arbitrary $\Gamma$. Along with
$\mathcal{X}_{\Gamma}GL_{n}$, we consider what we call the \emph{irreducible
character varieties:} 
\[
\mathcal{X}_{\Gamma}^{irr}GL_{n}\subset\mathcal{X}_{\Gamma}GL_{n},
\]
which are Zariski open subvarieties consisting of (equivalence classes
of) irreducible representations $\rho:\Gamma\to GL_{n}$. Let us denote
the plethystic exponential of a formal power series $f(x,y,z)\in\mathbb{Q}[x,y][[z]]$
by $\pexp(f)$ (definition in Section 4). We prove: 
\begin{thm}
\label{thm:main}Let $\Gamma$ be a finitely generated group. Then,
in $\mathbb{Q}[u,v][[t]]$: 
\[
\sum_{n\geq0}E(\mathcal{X}_{\Gamma}GL_{n};u,v)\,t^{n}=\pexp\left(\sum_{n\geq1}E(\mathcal{X}_{\Gamma}^{irr}GL_{n};u,u)\,t^{n}\right).
\]

\end{thm}
Unravelling the above power series, and the definitions and properties
of the plethystic functions, we obtain a closed formula for each individual
$E$-polynomial of $\mathcal{X}_{\Gamma}GL_{n}$ as a finite sum in
the $E$-polynomials of the irreducible character varieties $\mathcal{X}_{\Gamma}^{irr}GL_{n}$
of lower dimension, indexed by what we call \emph{rectangular partitions
of $n$} (see Definition \ref{def:rectangular-partition}). 
\begin{cor}
\label{cor:main}For every $n$ and $\Gamma$ as above, 
\[
E(\mathcal{X}_{\Gamma}GL_{n};u,v)=\sum_{[[k]]\in\mathcal{RP}_{n}}\ \prod_{l,h=1}^{n}\frac{E(\mathcal{X}_{\Gamma}^{irr}GL_{l};u^{h},v^{h})^{k_{l,h}}}{k_{l,h}!\,h^{k_{l,h}}},
\]
where $\mathcal{RP}_{n}$ is the (finite) set of all rectangular partitions
of $n$. 
\end{cor}
As a first application of these results, in Section 5 we write the
$E$-polynomial of the abelian stratum $\mathcal{X}_{\Gamma}^{[1^{n}]}GL_{n}\subset\mathcal{X}_{\Gamma}GL_{n}$
in terms of usual partitions, generalizing a result in \cite{FS};
we also apply the same methods to write the $E$-polynomial of the
so-called Cartan brane on the moduli space of rank $n$ and degree
zero Higgs bundles, an algebraic variety which is generally not of
polynomial type.

Theorem \ref{thm:main} and Corollary \ref{cor:main} work both ways
so that, by knowing all polynomials $E(\mathcal{X}_{\Gamma}GL_{m};u,v)$
for all $m\leq n$, we are able to determine $E(\mathcal{X}_{\Gamma}^{irr}GL_{n};u,v)$.
This is explored in the last subsection, where by using previous computations
of $E$-polynomials of $\mathcal{X}_{\Gamma}GL_{n}$, for $n=2$ and
$3$, and for groups $\Gamma$ other than the free group (mainly using
\cite{BH}), we determine $E$-polynomials of some \emph{irreducible
character varieties} that have not been calculated before: when $\Gamma$
is the fundamental group of a compact surface (in both the orientable,
and non-orientable cases) and when $\Gamma$ is a torus knot group.
From these formules, we readily obtain new results for these groups
$\Gamma$: the number of irreducible components of $\mathcal{X}_{\Gamma}^{irr}GL_{n}$
and their Euler characteristics.

\subsection*{Acknowledgements }

We would like to thank A. González-Prieto, E. Franco, S. Lawton, M.
Logares, J. Mart\'{i}nez, S. Mozgovoy, V. Muñoz, A. Oliveira, F. Rodriguez-Villedgas,
J. Silva and M. Tierz for several interesting and very useful conversations
on topics around mixed Hodge structures and $E$-polynomials. We also
thank the organizers of the VII Iberoamerican Congress on Geometry,
Valladolid (2018) and of the Special Session on Geometry of Representation
Spaces in the Joint AMS/MMA Meeting (2019), where preliminary versions
of these results were presented.

\section{Mixed Hodge structures and $E$-polynomials}

Let $X$ be a quasi-projective variety over $\mathbb{C}$ (possibly
singular, not complete, and/or not irreducible). Denote by $H_{c}^{k}(X):=H_{c}^{k}(X,\mathbb{C})$
its degree $k$ (singular) complex cohomology group, with compact
support, for $k\in\{0,\cdots,2d\}$, where $d$ is the complex dimension
of $X$. Deligne defined natural and functorial mixed Hodge structures
on the $H_{c}^{k}(X)$, which are subtle algebraic invariants of $X$
(c.f. \cite{De}). For the general theory of mixed Hodge structures
on cohomology groups and its properties, see \cite{De} and \cite{PS}.
Here, we review their most important features for our purposes, and
introduce the notation.

\subsection{Mixed Hodge polynomials}

Numerically, mixed Hodge structures on $X$ can be codified via the
so-called mixed Hodge numbers 
\[
h^{k,p,q}(X)=\dim_{\mathbb{C}}H_{c}^{k,p,q}(X)\in\mathbb{N}_{0},
\]
where $p,q\in\{0,\cdots,k\}$. We say that $(p,q)$ are $k$-weights
of $X$, when $h^{k,p,q}\neq0$.

In general, mixed Hodge numbers verify $h^{k,p,q}=h^{k,q,p}$, and
$\dim_{\mathbb{C}}H_{c}^{k}(X)=\sum_{p,q}h^{k,p,q}$, so they provide
the (compactly supported) Betti numbers (and the usual Betti numbers,
in the smooth case, by Poincaré duality). For some interesting classes
of spaces, the above sum reduces to a one-variable sum. For example,
when $X$ is a compact Kähler manifold, the Hodge structure is called
\emph{pure}, which means that for each $k$, the only $k$-weights
are of the form $(p,k-p)$ with $p\in\{0,\cdots,k\}$. Another such
case, relevant for the present article, is when $X$ is of \emph{Hodge-Tate
type} (also called \emph{balanced} \emph{type}), for which all the
$k$-weights are of the form $(p,p)$ with $p\in\{0,\cdots,k\}$.

We can assemble all the $h^{k,p,q}(X)$ in the mixed Hodge polynomial
\begin{equation}
\mu(X;\,t,u,v):=\sum_{k,p,q\geq0}h^{k,p,q}(X)\ t^{k}u^{p}v^{q}\in\mathbb{\mathbb{N}}_{0}[t,u,v],\label{eq:mu}
\end{equation}
of three variables. The mixed Hodge polynomial specializes to the
(compactly supported) Poincaré polynomial by setting $u=v=1$, $P_{t}^{c}(X)=\mu(X;\,t,1,1).$
Again, this gives the usual Poincaré polynomial in the smooth situation.

\subsection{The $E$-polynomial}

Mixed Hodge polynomials are generally difficult to compute. However,
by substituting $t=-1$ we obtain a certain Euler characteristic version,
which is easier to compute due to its multiplicative and additive
properties. We define the \emph{$E$-polynomial} of $X$ by 
\begin{equation}
E(X;\,u,v)=\sum_{k,p,q}(-1)^{k}h^{k,p,q}(X)\ u^{p}v^{q}\in\mathbb{Z}[u,v],\label{eq:h_c}
\end{equation}
which is also called the $E$-polynomial. Observe that 
\[
\chi^{c}(X)=E(X;\,1,1)=\mu(X;\,-1,1,1)
\]
is the (compactly supported) Euler characteristic of $X$.

The Künneth theorem is valid for mixed Hodge structures (see \cite{PS})
and so, $\mu$ verifies a multiplicative property with respect to
Cartesian products: 
\[
\mu(X\times Y)=\mu(X)\mu(Y),
\]
and induces analogous statements for $P^{c}$ and $E$ (we write simply
$\mu(X)$, $P^{c}(X)$, $E(X)$ etc, in formulae where the variables
of the polynomials are not relevant).

The big computational advantage of $E(X)$, as compared to $\mu(X)$
or $P^{c}(X)$ is that it satisfies both an \emph{additive property
with respect to stratifications} by locally closed (in the Zariski
topology) strata and a \emph{multiplicative property for fibrations}
in at least three important situations that we summarize in the following
statement. 
\begin{prop}
\label{prop:Properties-of-E} \cite{DL,LMN} If the quasi-projective
variety $X$ has a closed subvariety $Z\subset X$ (so that $X=Z\sqcup(X\setminus Z)$
is a stratification of $X$ by locally closed subvarieties), then
\[
E(X)=E(Z)+E(X\setminus Z).
\]
Also, if $X$ is the total space of an algebraic fibration of quasi-projective
varieties 
\[
F\to X\to B,
\]
and either:\\
 (i) it is locally trivial in the Zariski topology of $B$, or\\
 (ii) $F$, $X$ and $B$ are smooth, the fibration is locally trivial
in the complex analytic topology, and $\pi_{1}(B)$ acts trivially
on $H_{c}^{*}(F)$, or\\
 (iii) $X$, $B$ are smooth and $F$ is a complex connected Lie group.\\
 Then 
\[
E(X)=E(F)\cdot E(B).
\]
\end{prop}
\begin{proof}
The additive property is well known and can be found in \cite{DL}
or in the book \cite{PS}. The multiplicative property presented here
is a slight reformulation (in the non-equivariant case) of the one
in Dimca-Lehrer \cite[Thm. 6.1]{DL} (and \cite[Remarks 6.2]{DL}),
and also appears in \cite[Prop. 1.9]{LMN}; a more detailed proof
has been recently presented in \cite{FS}, so we refer to those proofs,
adding only a couple of comments that may serve to deduce the present
statement.

The weight polynomial used by Dimca-Lehrer is equivalent to the $E$-polynomial
in the case of Hodge-Tate type varieties, using the substitution $t^{2}=uv$.
So, this statement is a generalization of \cite[Thm. 6.1]{DL} to
the 2 variable $E$-polynomial. Note also that the case (iii) actually
follows from (ii) since the action of $\pi_{1}(B)$ on the cohomology
of a connected Lie group $F$ is always trivial.\end{proof}
\begin{example}
\label{exa:E}(1) Let $n\in\mathbb{N}_{0}$. Simple calculations give
\[
\mu(\mathbb{C}^{n})=t^{2n}u^{n}v^{n},\quad\mu(\mathbb{C}^{*};\,t,u,v)=t^{2}uv+t.
\]
This implies that $E(\mathbb{C}^{n})=(uv)^{n}$ and $E(\mathbb{C}^{*})=uv-1$,
a result compatible with the locally closed decomposition $\mathbb{C}=\mathbb{C}^{*}\sqcup\{0\}$.
Note the absense of additivity for $\mu$.\\
 (2) The group $GL_{n}\mathbb{C}$ can be given as the fibration of
smooth varieties, 
\[
SL_{n}\mathbb{C}\to GL_{n}\mathbb{C}\to\mathbb{C}^{*},
\]
whose projection map is the determinant. This is not locally trivial
in the Zariski topology, but it is so in the analytic topology, and
the fact that the complex Lie group $SL_{n}\mathbb{C}$ is connected
implies that $\pi_{1}(\mathbb{C}^{*})$ acts trivially on the cohomology
of $SL_{n}\mathbb{C}$. Then, the Proposition~\ref{prop:Properties-of-E}
implies: 
\[
E(GL_{n}\mathbb{C})=E(SL_{n}\mathbb{C})(uv-1).
\]
Note that all the groups in the fibration are of Hodge-Tate type.\footnote{In fact, every complex algebraic reductive group $G$ is of Hodge-Tate
type (see, eg. \cite{DL} and \cite{Jo}).} 
\end{example}
In this article, if the $E$-polynomial of an algebraic variety $X$
depends only on the product $uv$ (for example, when $X$ is of Hodge-Tate
type, such as the cases in Example \ref{exa:E}), we write $x=uv$
and use the notation: 
\[
E_{x}(X):=E(X;\,\sqrt{x},\sqrt{x})\in\mathbb{Z}[x].
\]
For example, since $E(\mathbb{C}^{*};\,u,v)=uv-1$ we write $E_{x}(\mathbb{C}^{*})=x-1$.
Then $E_{x}((\mathbb{C}^{*})^{l})=(x-1)^{l}$, by the product formula.


\section{Affine GIT and Character Varieties}

In this section we recall some aspects of Geometric Invariant Theory
(GIT) and of character varieties of finitely presented groups.

\subsection{Affine GIT}

Consider an affine algebraic variety $X$ over $\mathbb{C}$, and
an affine algebraic reductive $\mathbb{C}$-group $G$. Given an algebraic
action of $G$ on $X$, we have an induced action of $G$ on the ring
$\mathbb{C}[X]$ of regular functions on $X$ and we can define the
(affine) GIT quotient by 
\[
X\quot G:=Spec\left(\mathbb{C}[X]^{G}\right),
\]
where $\mathbb{C}[X]^{G}$ denotes the subring of $G$-invariants
in $\mathbb{C}[X]$. In many situations this quotient differs from
the usual orbit quotient, since this one identifies $G$-orbits whose
closures intersect. Nevertheless, with the notion of stability we
can sometimes recover good properties of the GIT quotient. Let $G_{x}\subset G$
denote the stabilizer of a point $x\in X$ and let us call the subgroup
$G_{X}:=\cap_{x\in X}G_{x}$ the \emph{center of the action}, since
it acts trivially, and $G/G_{X}$ acts effectively on $X$. Denote
by $\psi_{x}$ be the (effective) orbit map through $x$: 
\begin{eqnarray*}
\psi_{x}:G/G_{X} & \to & X\\
g & \mapsto & g\cdot x.
\end{eqnarray*}

\begin{defn}
\label{def:polystable}In the situation above, we say that $x\in X$
is \emph{polystable} if the orbit $G\cdot x$ is closed in $X$. We
say that $x\in X$ is \emph{stable} if it is polystable and $\psi_{x}$
is a proper map.\end{defn}
\begin{rem}
\label{rem:stability}This definition of stability differs from that
of \cite{MJK}, being equivalent to the more common notion when $G_{X}$
is finite (see \cite{Ki,CF}). The above definition is more convenient
in this article (as was the case in \cite{Ki}) since, for character
varieties, $G_{X}$ always contains the center of $G$ (see below). 
\end{rem}
By standard GIT results, one can show that the stable locus $X^{s}\subset X$
is a Zariski open (hence dense, when non-empty) set and one gets a
better quotient for the stable locus. We say that a morphism 
\[
f:X\to Y
\]
is a \emph{geometric quotient} if $f$ is $G$-invariant, induces
the quotient topology on $Y$, and it is a bijection $Y=X/G$ which
preserves rings of functions in the sense that $\mathbb{C}[f^{-1}(U)]^{G}=\mathbb{C}[U],$
for every $U\subset Y$ open. The following shows that the stable
quotient is geometric. 
\begin{prop}
\label{prop:geo-quotient}The restriction $X^{s}\to X^{s}/G$ of the
affine quotient map $\Phi:X\to X\quot G$ is a geometric quotient.
Moreover, $\Phi(X^{s})$ is Zariski open in $X\quot G$.\end{prop}
\begin{proof}
See \cite[Chap. 5]{Mu}. 
\end{proof}

\subsection{Character varieties}

Let $G$ be as before, and let $\Gamma$ be a finitely presented group.
Denote by 
\[
\mathcal{R}_{\Gamma}G=\hom(\Gamma,G)
\]
the algebraic variety of representations of $\Gamma$ in $G$. An
element $\rho\in\mathcal{R}_{\Gamma}G$ is defined by $\rho(\gamma)$,
for $\gamma$ in a generating set for $\Gamma$, and the elements
$\rho(\gamma)\in G$, satisfy the algebraic relations of $\Gamma$.
Consider also the algebraic action of $G$ on $\mathcal{R}_{\Gamma}G$
by conjugation of representations. The corresponding GIT quotient
is the $G$-character variety of $\Gamma$: 
\[
\mathcal{X}_{\Gamma}G:=\hom(\Gamma,G)\quot G,
\]
sometimes also called the moduli space of representations of $\Gamma$
into $G$.

We will need to work also with an alternative description of this
quotient, by using polystable representations which, according to
Definition \ref{def:polystable}, are representations $\rho\in\hom(\Gamma,G)=\mathcal{R}_{\Gamma}G$
whose orbits $G\cdot\rho:=\{g\rho g^{-1}:\ g\in G\}$ are (Zariski)
closed. The subset of polystable representations in $\mathcal{R}_{\Gamma}G$
is denoted by $\mathcal{R}_{\Gamma}^{ps}G$, and it can be shown that
$\mathcal{R}_{\Gamma}^{ps}G\subset\mathcal{R}_{\Gamma}G$ is a Zariski
locally-closed subvariety (containing the stable locus $\mathcal{R}_{\Gamma}^{s}G\subset\mathcal{R}_{\Gamma}G$,
but neither open nor closed in general). It can also be shown that
a representation $\rho:\Gamma\to G$ is polystable if and only if
it is completely reducible. This means that if $\rho(\Gamma)$ is
contained in some proper parabolic $P\subset G$, then it is actually
contained in a Levi subgroup of $P$ (see \cite{Si}). 
\begin{prop}
\cite{FL1}\label{prop:character-polystable representation} There
is a bijective correspondence: 
\[
\mathcal{X}_{\Gamma}G=\mathcal{R}_{\Gamma}G\quot G\cong\mathcal{R}_{\Gamma}^{ps}G/G,
\]
where the right hand side is called the polystable quotient. 
\end{prop}
We also need the notion of irreducible representations, and consider
the character varieties consisting of these ``nicer'' representations.
For a given $\rho\in\mathcal{R}_{\Gamma}^{ps}G$, denote by $Z_{\rho}:=G_{\rho}$
the centralizer of $\rho(\Gamma)$ inside $G$ (coincides with the
stabilizer of $\rho$). For character varieties, $Z_{\rho}$ always
contains $ZG$, the center of $G$. Hence, $ZG$ is always contained
in the center of the action (justifying our definition of stability,
when $\dim ZG>0$). 
\begin{defn}
Let $\rho\in\mathcal{R}_{\Gamma}^{ps}G$. We say that $\rho$ is \emph{irreducible}
if $Z_{\rho}$ is a finite extension of $ZG$.\end{defn}
\begin{rem}
(1) This definition is equivalent to the usual definition, involving
parabolic subgroups: indeed, $\rho\in\hom(\Gamma,G)$ is irreducible
if and only if it is polystable and its image is not contained in
a proper parabolic subgroup of $G$ (see \cite{CF,Si}). \\
 (2) For character varieties, irreducibility is equivalent to stability
in the sense of Definition \ref{def:polystable} (see \cite[Prop. 5.11 (iii)]{CF}).
So, the subset of irreducible representations, denoted $\mathcal{R}_{\Gamma}^{irr}G\subset\mathcal{R}_{\Gamma}^{ps}G$,
equals the stable locus, and being a Zariski open subset of $\mathcal{R}_{\Gamma}G$,
is a quasi-projective variety. 
\end{rem}
Since irreducibility is well defined on $G$-orbits, we define the\emph{
$G$-irreducible character variety of $\Gamma$} as 
\[
\mathcal{X}_{\Gamma}^{irr}G:=\mathcal{R}_{\Gamma}^{irr}G/G
\]
which is a geometric quotient. Hence, $\mathcal{X}_{\Gamma}^{irr}G$
is a Zariski open subvariety of $\mathcal{X}_{\Gamma}G$, by Proposition
\ref{prop:geo-quotient}.

\subsection{Stratification by stabilizer dimension}

Let $G_{\mathcal{R}}:=G_{\mathcal{R}_{\Gamma}G}\subset G$ be the
center of the action of $G$ on $\mathcal{R}_{\Gamma}G$, that is:
\[
G_{\mathcal{R}}:=\bigcap_{\rho\in\mathcal{R}_{\Gamma}G}Z_{\rho},
\]
where $Z_{\rho}$ is the stabilizer of $\rho$. Then, $\dim Z_{\rho}\geq\dim G_{\mathcal{R}}\geq\dim ZG$,
for all $\rho\in\mathcal{R}_{\Gamma}G$. 
\begin{prop}
\label{prop:locally-closed-general-G}Let $m_{0}:=\dim G_{\mathcal{R}}\in\mathbb{N}_{0}$.
Then, the character variety $\mathcal{X}_{\Gamma}G$ can be written
as a union of locally closed quasi-projective varieties, 
\[
\mathcal{X}_{\Gamma}G=\bigsqcup_{m\geq m_{0}}\mathcal{X}_{\Gamma}^{m}G,
\]
where $\mathcal{X}_{\Gamma}^{m}G$ consists of equivalence classes
of polystable representations $\rho$ with $\dim Z_{\rho}=m$. Moreover,
$\mathcal{X}_{\Gamma}^{m_{0}}G$ is precisely the open and dense stable
locus $\mathcal{X}_{\Gamma}^{s}G=\Phi(\mathcal{R}_{\Gamma}^{s}G)$
as in Proposition \ref{prop:geo-quotient}.\end{prop}
\begin{proof}
Let $\mathcal{R}_{\Gamma}^{m}G\subset\mathcal{R}_{\Gamma}^{ps}G$
be the subset of all polystable representations $\rho\in\mathcal{R}_{\Gamma}^{ps}G$
such that $\dim Z_{\rho}=m$, and note that we have 
\begin{eqnarray}
\mathcal{R}_{\Gamma}^{ps}G & = & \bigsqcup_{m\geq m_{0}}\mathcal{R}_{\Gamma}^{m}G,\label{eq:ps-disjoint-union1}
\end{eqnarray}
as a finite set-theoretic disjoint union. Since the stabilizer dimension
is a conjugation invariant, denote their equivalence classes under
conjugation by $\mathcal{X}_{\Gamma}^{m}G=\mathcal{R}_{\Gamma}^{m}G\quot G$.
By Proposition~\ref{prop:character-polystable representation}, the
character variety $\mathcal{X}_{\Gamma}G$ is isomorphic to the polystable
quotient and, hence, equation~\eqref{eq:ps-disjoint-union1} yields
the set-theoretic disjoint union: 
\begin{eqnarray}
\mathcal{X}_{\Gamma}G & \cong & \bigsqcup_{m\geq m_{0}}\mathcal{R}_{\Gamma}^{m}G\quot G=\bigsqcup_{m\geq m_{0}}\mathcal{X}_{\Gamma}^{m}G.
\end{eqnarray}
To prove the locally closedness property, consider the following construction.
Observe that the subset $R_{\Gamma}^{m_{0}}G=\{\rho\in\mathcal{R}_{\Gamma}^{ps}G|\dim Z_{\rho}=m_{0}=\dim G_{\mathcal{R}}\}$
(which is non-empty by assumption) is precisely the subset of stable
points, since the condition $\dim Z_{\rho}=m_{0}$ is equivalent to
$\dim Z_{\rho}$ being minimal. This also means that the orbit map
$\psi_{\rho}$ is proper and conversely (see Definition \ref{def:polystable}),
proving the last statement. Therefore, from Proposition \ref{prop:geo-quotient},
the restriction 
\[
\Phi_{m_{0}}:R_{\Gamma}^{m_{0}}G\to R_{\Gamma}^{m_{0}}G/G
\]
is a geometric quotient, and the stable locus $R_{\Gamma}^{m_{0}}G\subset\mathcal{R}_{\Gamma}^{ps}G$
and $\mathcal{X}_{\Gamma}^{m_{0}}G:=\Phi_{m_{0}}(R^{m_{0}})\subset\mathcal{X}_{\Gamma}G$
are Zariski open subsets. Now, let $R_{\Gamma}^{>m_{0}}G:=\mathcal{R}_{\Gamma}^{ps}G\setminus R_{\Gamma}^{m_{0}}G$.
Given that $R_{\Gamma}^{>m_{0}}G$ is Zariski closed in $\mathcal{R}_{\Gamma}^{ps}G$
and the action of $G$ is well defined on it, we can repeat the argument
for the subset: 
\[
R_{\Gamma}^{m_{1}}G:=\{\rho\in R_{\Gamma}^{>m_{0}}G\,|\,\dim Z_{\rho}=m_{1}\}\subset R_{\Gamma}^{>m_{0}}G,
\]
where $m_{1}\in\mathbb{N}$ is the minimum of the dimensions of $\{Z_{\rho}\,|\,\rho\in R_{\Gamma}^{>m_{0}}G\}$;
then $R_{\Gamma}^{m_{1}}G$ is a Zariski open (and non-empty) subset
(of the Zariski closed set $R_{\Gamma}^{>m_{0}}G$) containing all
stable representations in $R_{\Gamma}^{>m_{0}}G$. Hence, again, the
restriction 
\[
\Phi_{m_{1}}:R_{\Gamma}^{m_{1}}G\to R_{\Gamma}^{m_{1}}G/G,
\]
is a geometric quotient and $\mathcal{X}_{\Gamma}^{m_{1}}G:=\Phi_{m_{1}}(R_{\Gamma}^{m_{1}}G)$
is also an open subset of the Zariski closed set $\mathcal{X}_{\Gamma}G\setminus\mathcal{X}_{\Gamma}^{m_{0}}G=R_{\Gamma}^{>m_{0}}G\quot G$,
and therefore $\mathcal{X}_{\Gamma}^{m_{1}}G$ is locally closed.
By repeating this procedure in a finite number of steps we obtain
a stratification of the character variety $\mathcal{X}_{\Gamma}G$
by locally-closed quasi-projective varieties, which completes the
proof. 
\end{proof}

\section{The linear case: $GL_{n}$.}

\label{section:linearcase}

In this Section, we examine the linear case, the case of $G=GL_{n}$.

Let $\Gamma$ be a finitely presented group. We now provide explicit
formulae for the $E$-polynomials of $GL_{n}$-character varieties
of $\Gamma$ in terms of $E$-polynomials of all irreducible $GL_{m}$-character
varieties of $\Gamma$, for $m\leq n$. These formulae present several
interesting features: firstly, they are independent of the group $\Gamma$;
secondly they relate, not just the individual polynomials $E(\mathcal{X}_{\Gamma}GL_{n})$,
but their generating functions (as power series in a formal variable),
to the corresponding generating functions of the $E(\mathcal{X}_{\Gamma}^{irr}GL_{n})$;
moreover, the relation between these two kinds of generating functions
is the so-called \emph{plethystic exponential} which plays a prominent
role in the combinatorics of symmetric functions, and has applications
in counting of gauge invariant operators in supersymmetric quantum
theories (see eg. \cite{FHH}).

Note that, besides their intrinsic relevance, irreducible character
varieties often coincide, or are related with, the smooth locus of
the full character varieties. For example, by \cite{FL2} the irreducible
character variety $\mathcal{X}_{\Gamma}^{irr}GL_{n}$ coincides precisely
with the \emph{smooth locus} of the full character variety $\mathcal{X}_{\Gamma}GL_{n}$,
in the case of the free group $\Gamma=F_{r}$. Recently, this theme
has been greatly expanded in \cite{GLR}.

\subsection{The stratification by partition type }

We start by describing what we call the\emph{ stratification by partition
type} of our $GL_{n}$-character varieties, a convenient refinement
of the stratification by stabilizer dimension of Proposition~\ref{prop:locally-closed-general-G}.
Given the standard representation of $GL_{n}$ in $\mathbb{C}^{n}$,
we have a natural notion of direct sum $\rho_{1}\oplus\rho_{2}\in\mathcal{R}_{\Gamma}GL_{n_{1}+n_{2}}$
of representations $\rho_{i}:\Gamma\to GL_{n_{i}}$, $i=1,2$. This
is clearly a commutative operation.

To proceed, we need to consider partitions of $n$, and employ the
following ``power'' notation. A partition of $n\in\mathbb{N}$ is
denoted by $[k]=[1^{k_{1}}\cdots j^{k_{j}}\cdots n^{k_{n}}]$ where
the exponent $k_{j}$ means that $[k]$ has $k_{j}\geq0$ parts of
size $j\in\{1,\cdots,n\}$, so that $n=\sum_{j=1}^{n}j\cdot k_{j}$.
The sum of the exponents $|[k]|:=\sum k_{j}$ will be called the \emph{length}
of $[k]$ and $\mathcal{P}_{n}$ stands for the finite set of partitions
of $n\in\mathbb{N}$. For example, $[1^{2}\,4]\in\mathcal{P}_{6}$
is the partition $6=4+1+1$, whose length is $3$. 
\begin{defn}
Let $G=GL_{n}$ and $[k]\in\mathcal{P}_{n}$. We say that $\rho\in\mathcal{R}_{\Gamma}G=\hom(\Gamma,G)$
is $[k]$-polystable if $\rho$ is conjugated to 
\begin{equation}
\bigoplus_{j=1}^{n}\rho_{j}\label{eq:polystable-type}
\end{equation}
where each $\rho_{j}$ is, in turn, a direct sum of $k_{j}>0$ \emph{irreducible}
representations of $\mathcal{R}_{\Gamma}(GL_{j})$, for $j=1,\cdots,n$
(by convention, if some $k_{j}=0$, then $\rho_{j}$ is not present
in the direct sum). We denote $[k]$-polystable representations by
$\mathcal{R}_{\Gamma}^{[k]}G$ and use similar terminology/notation
for equivalence classes under conjugation $\mathcal{X}_{\Gamma}^{[k]}G\subset\mathcal{X}_{\Gamma}G$. \end{defn}
\begin{rem}
\label{irrstrata} We note that the trivial partition $[n]=[n^{1}]$
(of minimal length 1) corresponds exactly to the irreducible locus:
$\mathcal{R}_{\Gamma}^{[n]}G=\mathcal{R}_{\Gamma}^{irr}G$ and $\mathcal{X}_{\Gamma}^{[n]}G=\mathcal{X}_{\Gamma}^{irr}G$.
Moreover, $\mathcal{R}_{\Gamma}^{[k]}G\subset\mathcal{R}_{\Gamma}^{ps}G$
as every $[k]$-polystable representation, being a sum of irreducibles,
has indeed a closed $G$-orbit inside $\mathcal{R}_{\Gamma}G$.\end{rem}
\begin{prop}
\label{prop:locally-closed-GLn}Fix $n\in\mathbb{N}$, and let $G=GL_{n}$.
The character variety $\mathcal{X}_{\Gamma}G$ can be written as a
disjoint union, labelled by partitions $[k]\in\mathcal{P}_{n}$, of
locally closed quasi-projective varieties of $[k]$-polystable equivalence
classes: 
\[
\mathcal{X}_{\Gamma}G=\bigsqcup_{[k]\in\mathcal{P}_{n}}\mathcal{X}_{\Gamma}^{[k]}G,
\]
and this stratification refines the one by stabilizer dimension (Proposition
\ref{prop:locally-closed-general-G}).\end{prop}
\begin{proof}
Let $[k]=[1^{k_{1}}\cdots n^{k_{n}}]$ be a partition of $n$. As
in the proof of Proposition~\ref{prop:locally-closed-general-G}
note that 
\begin{eqnarray}
\mathcal{R}_{\Gamma}^{ps}G & = & \bigsqcup_{[k]\in\mathcal{P}_{n}}\mathcal{R}_{\Gamma}^{[k]}G,\label{eq:ps-disjoint-union}
\end{eqnarray}
is a set theoretic disjoint union, and the analogous decomposition
is valid for the polystable character variety $\mathcal{X}_{\Gamma}^{ps}G$.
Indeed, every polystable representation is completely reducible and
this means, for $G=GL_{n}$, that it is a direct sum of irreducibles.
Thus, for every $\rho\in\mathcal{R}_{\Gamma}^{ps}G$ we have a \emph{unique}
partition $[k]=[1^{k_{1}(\rho)}\cdots n^{k_{n}(\rho)}]\in\mathcal{P}_{n}$
so that $k_{j}(\rho)$ is the number (possibly zero) of representations
in $\mathcal{R}_{\Gamma}^{irr}GL_{j}$ that appear in the decomposition
of $\rho$ in \eqref{eq:polystable-type}. In particular, the length
of $[k]$, $\sum_{j}k_{j}(\rho)=|[k]|$, is well defined by $\rho\in\mathcal{R}_{\Gamma}^{ps}G$.
Proposition~\ref{prop:character-polystable representation} and equation~\eqref{eq:ps-disjoint-union}
yield 
\begin{eqnarray}
\mathcal{X}_{\Gamma}G & \cong & \bigsqcup_{[k]\in\mathcal{P}_{n}}\mathcal{R}_{\Gamma}^{[k]}G\quot G=\bigsqcup_{[k]\in\mathcal{P}_{n}}\mathcal{X}_{\Gamma}^{[k]}G,
\end{eqnarray}
as a set theoretic disjoint union. The fact that this is a stratification
by locally-closed quasi-projective varieties follows the same steps
of the proof in Proposition \ref{prop:locally-closed-general-G},
noting that, for every $\rho\in\mathcal{R}_{\Gamma}^{[k]}G$, we have
\[
\dim Z_{\rho}=\sum k_{j}(\rho)=|[k]|.
\]
Indeed, by Schur's lemma the stabilizer of $\rho_{j}\in\mathcal{R}_{\Gamma}^{irr}GL_{j}$
is the center of $GL_{j}$, which equals $\mathbb{C}^{*}$, independently
of $j>0$. Hence, this stratification refines the one in Proposition
\ref{prop:locally-closed-general-G} and different partitions with
the same length become (disjoint) irreducible components of each stratum
by stabilizer dimension. \end{proof}
\begin{cor}
Let $G=GL_{n}\mathbb{C}$. If, for a given character variety $\mathcal{X}_{\Gamma}G$,
all $[k]$-polystable strata for $[k]\in\mathcal{P}_{n}$ are of Hodge-Tate
type, then $\mathcal{X}_{\Gamma}G$ is of Hodge-Tate type. \end{cor}
\begin{proof}
This follows at once, by combining Proposition \ref{prop:locally-closed-GLn}
with Proposition~\ref{prop:Properties-of-E}. 
\end{proof}
Note that the converse statement is not valid in general. Moreover,
there are character varieties which are not of Hodge-Tate type. Indeed,
a recent article by I. Rapinchuk \cite{Ra} showed that every irreducible
affine variety, defined over $\mathbb{Q}$, can be written as an irreducible
component of a character variety; this class certainly contains varieties
which are not of Hodge-Tate type, such as a smooth affine cubic in
the plane isomorphic to an elliptic curve with one point removed\footnote{We thank Sean Lawton for providing us this reference.}.

\subsection{Generating functions of $E$-polynomials}

Recall that the partition $[k]=[1^{k_{1}}\cdots n^{k_{n}}]\in\mathcal{P}_{n}$
has $k_{j}\geq0$ parts of size $j\in\{1,\cdots,n\}$. For each $[k]\in\mathcal{P}_{n}$,
denote by $L_{[k]}$ the reductive subgroup: 
\begin{equation}
L_{[k]}:=GL_{1}^{k_{1}}\times\cdots\times GL_{n}^{k_{n}}\subset GL_{n},\label{eq:k-Levi}
\end{equation}
which we call the $[k]$-Levi of $GL_{n}$ (in fact, all Levi subgroups
of $GL_{n}$ are conjugate to one obtained in this way).

Now, $L_{[k]}$ acts naturally, factorwise, on the space of polystable
representations of type $[k]$, $\mathcal{R}_{\Gamma}^{[k]}G$, and
the GIT quotient is a product of irreducible character varieties (recall
that each block of polystable representations corresponds to irreducible
ones): 
\begin{equation}
\mathcal{R}_{\Gamma}^{[k]}G\quot L_{[k]}=(\mathcal{X}_{\Gamma}^{irr}GL_{1})^{k_{1}}\times(\mathcal{X}_{\Gamma}^{irr}GL_{2})^{k_{2}}\times\cdots\times(\mathcal{X}_{\Gamma}^{irr}GL_{n})^{k_{n}}.\label{eq:P/L}
\end{equation}
Note, however, that this does not coincide with the $[k]$-character
variety $\mathcal{X}_{\Gamma}^{[k]}GL_{n}$ as defined in Proposition
\ref{prop:locally-closed-GLn}. Indeed, when some $k_{j}>1$, there
is a permutation group acting on $\mathcal{R}_{\Gamma}^{[k]}G$ by
permuting the blocks of equal size. To obtain $\mathcal{X}_{\Gamma}^{[k]}G$
define, for each partition $[k]\in\mathcal{P}_{n}$, the finite subgroup
\[
S_{[k]}:=S_{k_{1}}\times S_{k_{2}}\times\cdots\times S_{k_{n}}\subset S_{n},
\]
of the symmetric group $S_{n}$ on $n$ letters. For an algebraic
variety $X$, we let $\sym^{m}(X)=X^{m}/S_{m}$ denote the $m^{th}$
symmetric product of $X$. 
\begin{prop}
\label{prop:sym-product} Let $G=GL_{n}$ and let $\Gamma$ be a finitely
presented group. For every partition $[k]\in\mathcal{P}_{n}$, there
are isomorphisms of algebraic varieties: 
\[
\mathcal{X}_{\Gamma}^{[k]}G\cong\times_{j=1}^{n}\sym^{k_{j}}(\mathcal{X}_{\Gamma}^{irr}GL_{j}).
\]
\end{prop}
\begin{proof}
This follows directly from the construction above. Indeed, since:
\[
\mathcal{X}_{\Gamma}^{[k]}G\cong\mathcal{R}_{\Gamma}^{[k]}G\quot G,
\]
and the action of $G$ on $\mathcal{R}_{r}^{[k]}G$ reduces to an
action of $L_{[k]}$ and the action of permutation of blocks of equal
size, we get from equation \eqref{eq:P/L}: 
\[
\mathcal{X}_{\Gamma}^{[k]}G\cong\left(\mathcal{R}_{\Gamma}^{[k]}G\quot L_{[k]}\right)/S_{[k]}\cong\left(\times_{j=1}^{n}(\mathcal{X}_{\Gamma}^{irr}GL_{j})^{k_{j}}\right)/\left(\times_{j=1}^{n}S_{k_{j}}\right).
\]
Moreover, since each subgroup $S_{k_{j}}\subset S_{n}$ only permutes
the $k_{j}$ blocks of size $j$, and does not act on other blocks,
the result follows from $X^{k}/S_{k}=\sym^{k}X$. 
\end{proof}
By the above proposition, we need to consider symmetric products of
irreducible character varieties. It is interesting to observe that
the $E$-polynomials of symmetric products are intrinsically related
to the so-called \emph{plethystic exponential functions}, which we
now recall. Given a power series $f\in\mathbb{Q}[x,y][[z]]$, formal
in $z$, written in the form: 
\begin{equation}
f(x,y,z)=\sum_{n\geq0}f_{n}(x,y)\,z^{n},\label{eq:formal-ps}
\end{equation}
where $f_{n}(x,y)\in\mathbb{Q}[x,y]$ are polynomials in $x,y$, with
rational coefficients\footnote{For our purposes, coefficients in $\mathbb{Q}$ are enough, although
the theory can be developed over any field or even ring. }, the plethystic exponential, denoted $\mbox{PExp}$, is defined formally
(in terms of the usual exponential) as: 
\[
\pexp(f):=e^{\Psi(f)}\in\mathbb{Q}[x,y][[z]],
\]
where $\Psi$, called the (multi-variable) \emph{Adams operator},
is the invertible $\mathbb{Q}$-linear operator on $\mathbb{Q}[x,y][[z]]$
acting on monomials in $x$, $y$ and $z$ as: $\Psi(x^{i}y^{j}z^{k})=\sum_{l\geq1}\frac{x^{li}y^{lj}z^{lk}}{l},$
where $(i,j,k)\in\mathbb{N}_{0}^{3}\setminus\{(0,0,0)\}$. Note that,
from the additivity of $\Psi$, we get the property: 
\[
\pexp(f_{1}+f_{2})=\pexp(f_{1})\,\pexp(f_{2}),\quad\forall f_{1},f_{2}\in\mathbb{Q}[x,y][[z]].
\]

\begin{prop}
\label{prop:Sym-Pexp}Let $X$ be a quasi-projective variety. Then,
the generating function of the $E$-polynomial of its symmetric products
$\mbox{Sym}^{m}(X)$, $m\in\mathbb{N}$, is a rational function, and
can be written as: 
\[
\sum_{n\geq0}E(\sym^{n}(X);u,v)\,y^{n}=\pexp(E(X;u,v)y).
\]
\end{prop}
\begin{proof}
We apply the generating function of J. Cheah who showed, in \cite{Ch},
the following formula: 
\[
\sum_{n\geq0}\mu(\sym^{n}(X);t,u,v)\,y^{n}=\prod_{k,p,q\geq0}(1-(-t)^{k}u^{p}v^{q}y)^{(-1)^{k+1}h^{k,p,q}(X)},
\]
(recall from equation \eqref{eq:h_c} that $h^{k,p,q}(X)$ are the
Hodge-Deligne numbers of $X$ for cohomology with \emph{compact support}).
Since, by definition, $E(X;u,v)=\sum_{k,p,q\geq0}(-1)^{k}h^{k,p,q}(X)\,u^{p}\,v^{q}$,
the above equality becomes: 
\[
\sum_{n\geq0}E(\sym^{n}(X);u,v)\,y^{n}=\prod_{k,p,q\geq0}(1-u^{p}v^{q}y)^{(-1)^{k+1}h^{k,p,q}(X)}.
\]
The proof follows from the next Lemma, by using $a_{p,q}:=\sum_{k\geq0}(-1)^{k}h^{k,p,q}(X)$. 
\end{proof}
Recall that plethystic exponentials have also a \emph{product form}.
The following can be shown in much greater generality; we restrict
to the case at need, for simplicity. 
\begin{lem}
\label{lem:Pexp-prod}If $g(u,v)=\sum_{p,q\geq0}a_{p,q}u^{p}v^{q}$
for some $a_{p,q}\in\mathbb{Z}$, then: 
\[
\pexp(g(u,v)y)=\prod_{p,q\geq0}(1-u^{p}v^{q}y)^{-a_{p,q}}.
\]
\end{lem}
\begin{proof}
Taking the logarithm of the left hand side, we get: 
\begin{eqnarray*}
\Psi(g(u,v)y) & = & \sum_{k\geq1}\frac{g(u^{k},v^{k})y^{k}}{k}=\sum_{k\geq1}\sum_{p,q\geq0}\frac{a_{p,q}u^{kp}v^{kq}y^{k}}{k}=\\
 & = & \sum_{p,q\geq0}a_{p,q}\sum_{k\geq1}\frac{(u^{p}v^{q}y)^{k}}{k}=-\sum_{p,q\geq0}a_{p,q}\log(1-u^{p}v^{q}y)\\
 & = & \log\left(\prod_{p,q\geq0}(1-u^{p}v^{q}y)^{-a_{p,q}}\right),
\end{eqnarray*}
which is the logarithm of the right hand side.\end{proof}
\begin{rem}
As mentioned, the above product formula is valid more generally, and
there are analogous formulae for formal power series in any number
of variables. 
\end{rem}
We will also use the following property of formal power series. 
\begin{lem}
\label{lem:formal-product}Let $R$ be a ring and let $g_{n}\in R[[t]]$
be a sequence of formal power series written as 
\[
g_{n}(t)=\sum_{k\geq0}a_{k}^{(n)}t^{k},\quad n\in\mathbb{N},\quad a_{k}^{(n)}\in R.
\]
Then 
\[
\prod_{n\geq1}g_{n}(t^{n})=g_{1}(t)\,g_{2}(t^{2})\,g_{3}(t^{3})\cdots=\sum_{n\geq0}\sum_{[k]\in\mathcal{P}_{n}}a_{k_{1}}^{(1)}\cdots a_{k_{n}}^{(n)}t^{n}.
\]
\end{lem}
\begin{proof}
This follows by expanding 
\[
(\sum_{k\geq0}a_{k}^{(1)}t^{k})(\sum_{k\geq0}a_{k}^{(2)}t^{2k})(\sum_{k\geq0}a_{k}^{(3)}t^{3k})\cdots=\sum_{n\geq0}b_{n}t^{n},
\]
and noting that $b_{n}$ collects all terms of the form $a_{k_{1}}^{(1)}\cdots a_{k_{n}}^{(n)}$
such that $n=\sum_{j=1}^{n}j\,k_{j}$. These are precisely the partitions
of $n$. 
\end{proof}
We are now ready for the proof of Theorem \ref{thm:main}, as follows. 
\begin{thm}
\label{thm:A-PExpB}Let $\Gamma$ be any finitely presented group,
and write $A_{n}^{\Gamma}(u,v):=E(\mathcal{X}_{\Gamma}GL_{n};u,v)$,
$B_{n}^{\Gamma}(u,v):=E(\mathcal{X}_{\Gamma}^{irr}GL_{n};u,v)$. Then:
\[
\sum_{n\geq0}A_{n}^{\Gamma}(u,v)t^{n}=\pexp\left(\sum_{n\geq1}B_{n}^{\Gamma}(u,v)t^{n}\right).
\]
\end{thm}
\begin{proof}
From Proposition \ref{prop:sym-product}, and the multiplicative property
of $E$ we get, for the $[k]$-polystable stratum of $\mathcal{X}_{\Gamma}GL_{n}$:
\[
E(\mathcal{X}_{\Gamma}^{[k]}GL_{n};u,v)=a_{k_{1}}^{(1)}(u,v)\cdots a_{k_{n}}^{(n)}(u,v),
\]
where we define the polynomials 
\[
a_{k}^{(j)}(u,v):=E(\sym^{k}(\mathcal{X}_{\Gamma}^{irr}GL_{j});u,v)\in\mathbb{Z}[u,v].
\]
Since the $E$-polynomial is also additive, we get by Proposition
\ref{prop:locally-closed-GLn}, 
\[
E(\mathcal{X}_{\Gamma}GL_{n};u,v)=\sum_{[k]\in\mathcal{P}_{n}}E(\mathcal{X}_{\Gamma}^{[k]}GL_{n};u,v)=\sum_{[k]\in\mathcal{P}_{n}}a_{k_{1}}^{(1)}(u,v)\cdots a_{k_{n}}^{(n)}(u,v).
\]
Now, we form the generating function: 
\[
\sum_{n\geq0}A_{n}^{\Gamma}(u,v)t^{n}=\sum_{n\geq0}E(\mathcal{X}_{\Gamma}GL_{n};u,v)t^{n}=\sum_{n\geq0}\left(\sum_{[k]\in\mathcal{P}_{n}}a_{k_{1}}^{(1)}(u,v)\cdots a_{k_{n}}^{(n)}(u,v)\right)t^{n}
\]
which, by Lemma \ref{lem:formal-product} (with $R=\mathbb{Z}[u,v]$)
equals to the product, 
\[
g_{1}(u,v)(t)\,g_{2}(u,v)(t^{2})\,g_{3}(u,v)(t^{3})\cdot\cdots=\prod_{n\geq1}g_{n}(u,v)(t^{n}),
\]
where: 
\begin{eqnarray*}
g_{n}(u,v)(t) & := & \sum_{k\geq0}a_{k}^{(n)}(u,v)t^{k}=\sum_{k\geq0}E(\sym^{k}(\mathcal{X}_{\Gamma}^{irr}GL_{n});u,v)t^{k}\\
 & = & \pexp(E(\mathcal{X}_{\Gamma}^{irr}GL_{n};u,v)t)=\pexp(B_{n}^{\Gamma}(u,v)t),
\end{eqnarray*}
and the last line used Proposition \ref{prop:Sym-Pexp}. Finally,
we use the multiplicative property of plethystic exponentials to obtain:
\[
\sum_{n\geq0}A_{n}^{\Gamma}(u,v)t^{n}=\prod_{n\geq1}g_{n}(u,v)(t^{n})=\prod_{n\geq1}\pexp(B_{n}^{\Gamma}(u,v)\,t^{n})=\pexp(\sum_{n\geq1}B_{n}^{\Gamma}(u,v)\,t^{n}),
\]
as wanted.\end{proof}
\begin{cor}
\label{cor:A-PExpB}Assume that $\mathcal{X}_{\Gamma}^{irr}GL_{n}$
is of Hodge-Tate type. Then: 
\[
\sum_{n\geq0}A_{n}^{\Gamma}(x)\,t^{n}=\pexp\left(\sum_{n\geq1}B_{n}^{\Gamma}(x)\,t^{n}\right),
\]
with $A_{n}^{\Gamma}(x)=E_{x}(\mathcal{X}_{\Gamma}GL_{n})$ and $B_{n}^{\Gamma}(x)=E_{x}(\mathcal{X}_{\Gamma}^{irr}GL_{n})$.\end{cor}
\begin{rem}
(1) When $\Gamma$ is the free group, the above formula appears in
the proof of \cite[Thm 2.5]{MR}. So Corollary \ref{cor:A-PExpB}
generalizes it to the general Hodge-Tate case.\\
 (2) The case when $B_{n}^{\Gamma}(x)=0$ for $n\geq2$ is still interesting.
For example, using $B_{1}^{\Gamma}(x)=(x-1)^{r}$ in Corollary \ref{cor:A-PExpB},
we recover the $E$-polynomials of the $GL_{n}$-character varieties
of $\Gamma=\mathbb{Z}^{r}$, the free abelian group of rank $r$.
See \cite{FS} and Subsection \ref{subsection:abelian} below.\\
 (3) We thank S. Mogovoy for drawing our attention to his recent Preprint
\cite{Mo1}, where another method of approaching this Corollary is
suggested (cf, \cite[Thm. 1.2]{Mo1}), within a general framework
for counting isomorphism classes of objects in additive categories
over finite fields (which can be traced back to \cite{Mo2}), using
also Katz's Theorem \cite[Appendix]{HRV1}. However, our proof of
Theorem \eqref{thm:A-PExpB}, does not rely on counting points over
finite fields, and hence remains valid for character varieties (over
$\mathbb{C}$) which are not necessarily of Hodge-Tate type or of
polynomial count. See Subsection \ref{subsection:cartan_brane} for
an example. 
\end{rem}

\subsection{Rectangular partitions and the $E$-polynomial of each individual
strata}

A further combinatorial analysis of the plethystic exponential in
Theorem \ref{thm:A-PExpB} allows us to get an explicit formula relating
the polynomials $A_{n}^{\Gamma}(u,v)$ and $B_{n}^{\Gamma}(u,v)$
through a \emph{finite process}: indeed, for a fixed $n$, $A_{n}^{\Gamma}(u,v)$
only depends on $B_{m}^{\Gamma}(u,v)$ for $m\leq n$, and this can
be given in a concrete way using what we call \emph{rectangular partitions}.
Moreover, this also allows to obtain a closed expression for the $E$-polynomials
of each individual strata, which can easily be implemented algorithmically
using standard computer software.

We start by noting that, in the particular case when $f$ is of the
form $f(x,y,z)=g(x,y)z$ (so that $f_{1}=g$ in equation \eqref{eq:formal-ps},
the remaining terms being zero), the plethystic exponential can be
written in yet another useful form, in terms of usual partitions. 
\begin{lem}
\label{lem:PExp-partitions}For any $g(u,v)\in\mathbb{Q}[u,v]$, we
have 
\[
\pexp(g(u,v)y)=\sum_{n\geq0}\,\left(\sum_{[k]\in\mathcal{P}_{n}}\prod_{j=1}^{n}\frac{g(u^{j},v^{j})^{k_{j}}}{k_{j}!\ j^{k_{j}}}\right)y^{n}\;.
\]
\end{lem}
\begin{proof}
By direct computation, we have: 
\begin{eqnarray*}
\pexp(g(u,v)y) & = & \exp\left(\Psi(g(u,v)y)\right)=\exp\left(\sum_{j\geq1}\frac{g(u^{j},v^{j})y^{j}}{j}\right)\\
 & = & \prod_{j\geq1}\exp\left(\frac{g(u^{j},v^{j})y^{j}}{j}\right)=\prod_{j\geq1}\,\sum_{k\geq0}\frac{g(u^{j},v^{j})^{k}y^{jk}}{k!\,j^{k}}\\
 & = & \sum_{n\geq0}y^{n}\left(\sum_{[k]\in\mathcal{P}_{n}}\prod_{j=1}^{n}\frac{g(u^{j},v^{j})^{k_{j}}}{k_{j}!\ j^{k_{j}}}\right),
\end{eqnarray*}
where in the last expression we gather all terms that contribute to
$y^{n}$. Since these correspond to writing $n=\sum_{j=1}^{n}j\,k_{j}$,
they correspond to partitions of $n$. 
\end{proof}
To write $\mbox{PExp}$ of an arbitrary series $f(x,y,z)\in\mathbb{Q}[x,y][[z]]$
in a similar way, we also need to develop a theory of \emph{rectangular
partitions} of a positive integer. 
\begin{figure}[h]
\begin{tikzpicture}[scale=0.5] \draw (-20,5) -- (-19,5); \draw (-20,4) -- (-19,4); \draw (-20,5) -- (-20,4); \draw (-19,5) -- (-19,4); \draw (-19,5) -- (-18,5); \draw (-19,4) -- (-18,4); \draw (-19,5) -- (-19,4); \draw (-18,5) -- (-18,4); \draw (-18,5) -- (-17,5); \draw (-18,4) -- (-17,4); \draw (-18,5) -- (-18,4); \draw (-17,5) -- (-17,4); \node [above] at (-18.5,5) {$k_{3,1}=1$};
\draw (-13,5) -- (-12,5); \draw (-13,4) -- (-12,4); \draw (-13,5) -- (-13,4); \draw (-12,5) -- (-12,4); \draw (-12,5) -- (-11,5); \draw (-12,4) -- (-11,4); \draw (-12,5) -- (-12,4); \draw (-11,5) -- (-11,4); \draw (-13,3.7) -- (-12,3.7); \draw (-13,2.7) -- (-12,2.7); \draw (-13,3.7) -- (-13,2.7); \draw (-12,3.7) -- (-12,2.7); \node [above] at (-12,5) {$k_{2,1}=k_{1,1}=1$};
\draw (-7,5) -- (-6,5); \draw (-7,4) -- (-6,4); \draw (-7,5) -- (-7,4); \draw (-6,5) -- (-6,4); \draw (-7,4) -- (-6,4); \draw (-7,3) -- (-6,3); \draw (-7,4) -- (-7,3); \draw (-6,4) -- (-6,3); \draw (-7,3) -- (-6,3); \draw (-7,2) -- (-6,2); \draw (-7,3) -- (-7,2); \draw (-6,3) -- (-6,2); \node [above] at (-6.5,5) {$k_{1,3}=1$};
\draw (-2,5) -- (-1,5); \draw (-2,4) -- (-1,4); \draw (-2,5) -- (-2,4); \draw (-1,5) -- (-1,4); \draw (-2,4) -- (-1,4); \draw (-2,3) -- (-1,3); \draw (-2,4) -- (-2,3); \draw (-1,4) -- (-1,3); \draw (-2,2.7) -- (-1,2.7); \draw (-2,1.7) -- (-1,1.7); \draw (-2,2.7) -- (-2,1.7); \draw (-1,2.7) -- (-1,1.7); \node [above] at (-1.5,5) {$k_{1,2}=k_{1,1}=1$};
\draw (3,5) -- (4,5); \draw (3,4) -- (4,4); \draw (3,5) -- (3,4); \draw (4,5) -- (4,4); \draw (3,3.7) -- (4,3.7); \draw (3,2.7) -- (4,2.7); \draw (3,3.7) -- (3,2.7); \draw (4,3.7) -- (4,2.7); \draw (3,2.4) -- (4,2.4); \draw (3,1.4) -- (4,1.4); \draw (3,2.4) -- (3,1.4); \draw (4,2.4) -- (4,1.4); \node [above] at (3.5,5) {$k_{1,1}=3$};
\end{tikzpicture}

\caption{\label{fig:RP3}The five rectangular partitions of $n=3$. The gluing
map $\pi$ takes the first one to the Young diagram of the partition
$[3]$, the second one corresponds to $[1\,2]$ and the last three
to $[1^{3}]$.}
\end{figure}
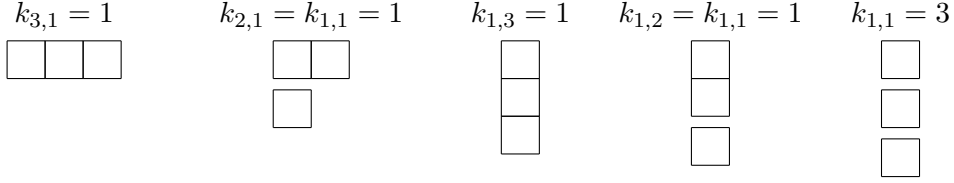

\begin{defn}
\label{def:rectangular-partition}Let $n\in\mathbb{N}$ be a natural
number. A \emph{rectangular partition} of $n$ is a double sequence
of non-negative integers $k_{l,h}\geq0$ (and $k_{l,h}\leq n$) for
each $l,h\in\{1,\cdots,n\}$ satisfying 
\[
n=\sum_{l=1}^{n}\sum_{h=1}^{n}l\,h\,k_{l,h},
\]
the finite set of rectangular paritions of $n$ is denoted by $\mathcal{RP}_{n}$
and such a rectangular partition is denoted by 
\[
[[k]]=[(1\times1)^{k_{1,1}}\,(1\times2)^{k_{1,2}}\cdots(1\times n)^{k_{1,n}}\cdots(n\times n)^{k_{n,n}}]\in\mathcal{RP}_{n}\;.
\]
There is a canonical ``gluing map'' sending a rectangular partition
to a usual partition: 
\begin{eqnarray*}
\pi:\mathcal{RP}_{n} & \to & \mathcal{P}_{n}\\{}
[[k]] & \mapsto & [m]=[1^{m_{1}}\cdots n^{m_{n}}]\quad\mbox{defined by }m_{l}:=\sum_{h=1}^{n}h\cdot k_{l,h}.
\end{eqnarray*}

\end{defn}
The geometric intepretation of rectangular partitions is as follows:
we are decomposing an initial set with area $n$, into a set of rectangles
of each possible size $l\times h\leq n$ (\emph{of length $l$ and
heigth $h$}), and each $l\times h$ rectangle appears with multiplicity
$k_{l,h}$ (rectangles $l\times h$ and $h\times l$ are considered
distinct). This explains the terminology ``gluing map'' as it is
obtained by gluing all rectangles to form the usual Young diagram
of a partition.%

\begin{example}
\label{exa:rectangular-partitions}For $n=3$, Figure \ref{fig:RP3}
shows the 5 possible rectangular partitions (all multiplicities $k_{l,h}$
not indicated are zero). Figure \ref{fig:RP4} shows the 11 cases
for $n=4$. 
\end{example}
The following general formula may be useful in other situations. 
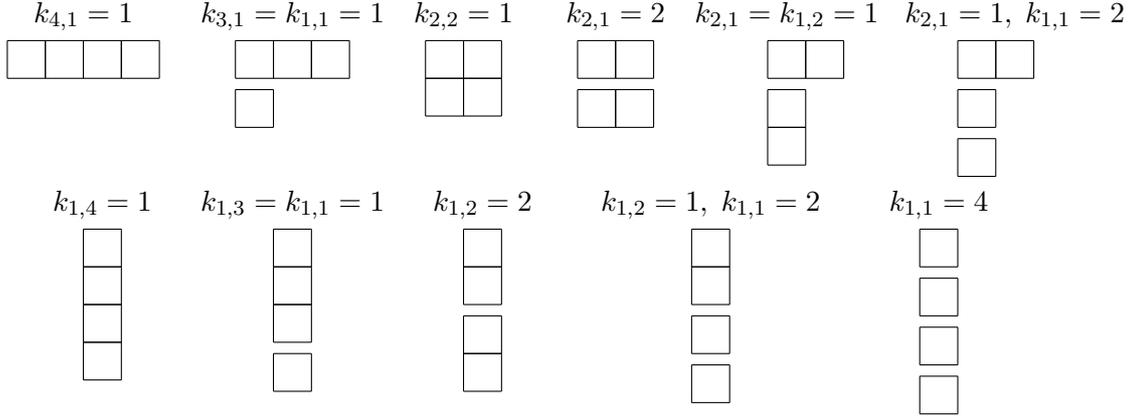
\begin{figure}[h]
\begin{tikzpicture}[scale=0.5]\draw (-20,5) -- (-19,5); \draw (-20,4) -- (-19,4); \draw (-20,5) -- (-20,4); \draw (-19,5) -- (-19,4); \draw (-19,5) -- (-18,5); \draw (-19,4) -- (-18,4); \draw (-19,5) -- (-19,4); \draw (-18,5) -- (-18,4); \draw (-18,5) -- (-17,5); \draw (-18,4) -- (-17,4); \draw (-18,5) -- (-18,4); \draw (-17,5) -- (-17,4); \draw (-17,5) -- (-16,5); \draw (-17,4) -- (-16,4); \draw (-17,5) -- (-17,4); \draw (-16,5) -- (-16,4); \node [above] at (-18,5) {$k_{4,1}=1$};
\draw (-14,5) -- (-13,5); \draw (-14,4) -- (-13,4); \draw (-14,5) -- (-14,4); \draw (-13,5) -- (-13,4); \draw (-13,5) -- (-12,5); \draw (-13,4) -- (-12,4); \draw (-13,5) -- (-13,4); \draw (-12,5) -- (-12,4); \draw (-12,5) -- (-11,5); \draw (-12,4) -- (-11,4); \draw (-12,5) -- (-12,4); \draw (-11,5) -- (-11,4); \draw (-14,3.7) -- (-13,3.7); \draw (-14,2.7) -- (-13,2.7); \draw (-14,3.7) -- (-14,2.7); \draw (-13,3.7) -- (-13,2.7); \node [above] at (-12.5,5) {$k_{3,1}=k_{1,1}=1$};
\draw (-9,5) -- (-8,5); \draw (-9,4) -- (-8,4); \draw (-9,5) -- (-9,4); \draw (-8,5) -- (-8,4); \draw (-8,5) -- (-7,5); \draw (-8,4) -- (-7,4); \draw (-8,5) -- (-8,4); \draw (-7,5) -- (-7,4); \draw (-9,4) -- (-8,4); \draw (-9,3) -- (-8,3); \draw (-9,4) -- (-9,3); \draw (-8,4) -- (-8,3); \draw (-8,4) -- (-7,4); \draw (-8,3) -- (-7,3); \draw (-8,4) -- (-8,3); \draw (-7,4) -- (-7,3); \node [above] at (-8,5) {$k_{2,2}=1$};
\draw (-5,5) -- (-4,5); \draw (-5,4) -- (-4,4); \draw (-5,5) -- (-5,4); \draw (-4,5) -- (-4,4); \draw (-4,5) -- (-3,5); \draw (-4,4) -- (-3,4); \draw (-4,5) -- (-4,4); \draw (-3,5) -- (-3,4); \draw (-5,3.7) -- (-4,3.7); \draw (-5,2.7) -- (-4,2.7); \draw (-5,3.7) -- (-5,2.7); \draw (-4,3.7) -- (-4,2.7); \draw (-4,3.7) -- (-3,3.7); \draw (-4,2.7) -- (-3,2.7); \draw (-4,3.7) -- (-4,2.7); \draw (-3,3.7) -- (-3,2.7); \node [above] at (-4,5) {$k_{2,1}=2$};
\draw (0,5) -- (1,5); \draw (0,4) -- (1,4); \draw (0,5) -- (0,4); \draw (1,5) -- (1,4); \draw (1,5) -- (2,5); \draw (1,4) -- (2,4); \draw (1,5) -- (1,4); \draw (2,5) -- (2,4); \draw (0,3.7) -- (1,3.7); \draw (0,2.7) -- (1,2.7); \draw (0,3.7) -- (0,2.7); \draw (1,3.7) -- (1,2.7); \draw (0,2.7) -- (1,2.7); \draw (0,1.7) -- (1,1.7); \draw (0,2.7) -- (0,1.7); \draw (1,2.7) -- (1,1.7); \node [above] at (0.5,5) {$k_{2,1}=k_{1,2}=1$};
\draw (5,5) -- (6,5); \draw (5,4) -- (6,4); \draw (5,5) -- (5,4); \draw (6,5) -- (6,4); \draw (6,5) -- (7,5); \draw (6,4) -- (7,4); \draw (6,5) -- (6,4); \draw (7,5) -- (7,4); \draw (5,3.7) -- (6,3.7); \draw (5,2.7) -- (6,2.7); \draw (5,3.7) -- (5,2.7); \draw (6,3.7) -- (6,2.7); \draw (5,2.4) -- (6,2.4); \draw (5,1.4) -- (6,1.4); \draw (5,2.4) -- (5,1.4); \draw (6,2.4) -- (6,1.4); \node [above] at (6.5,5) {$k_{2,1}=1,\; k_{1,1}=2$};
\draw (-18,0) -- (-17,0); \draw (-18,-1) -- (-17,-1); \draw (-18,0) -- (-18,-1); \draw (-17,0) -- (-17,-1); \draw (-18,-1) -- (-17,-1); \draw (-18,-2) -- (-17,-2); \draw (-18,-1) -- (-18,-2); \draw (-17,-1) -- (-17,-2); \draw (-18,-2) -- (-17,-2); \draw (-18,-3) -- (-17,-3); \draw (-18,-2) -- (-18,-3); \draw (-17,-2) -- (-17,-3); \draw (-18,-3) -- (-17,-3); \draw (-18,-4) -- (-17,-4); \draw (-18,-3) -- (-18,-4); \draw (-17,-3) -- (-17,-4); \node [above] at (-17.5,0) {$k_{1,4}=1$};
\draw (-13,0) -- (-12,0); \draw (-13,-1) -- (-12,-1); \draw (-13,0) -- (-13,-1); \draw (-12,0) -- (-12,-1); \draw (-13,-1) -- (-12,-1); \draw (-13,-2) -- (-12,-2); \draw (-13,-1) -- (-13,-2); \draw (-12,-1) -- (-12,-2); \draw (-13,-2) -- (-12,-2); \draw (-13,-3) -- (-12,-3); \draw (-13,-2) -- (-13,-3); \draw (-12,-2) -- (-12,-3); \draw (-13,-3.3) -- (-12,-3.3); \draw (-13,-4.3) -- (-12,-4.3); \draw (-13,-3.3) -- (-13,-4.3); \draw (-12,-3.3) -- (-12,-4.3); \node [above] at (-12.5,0) {$k_{1,3}=k_{1,1}=1$};
\draw (-8,0) -- (-7,0); \draw (-8,-1) -- (-7,-1); \draw (-8,0) -- (-8,-1); \draw (-7,0) -- (-7,-1); \draw (-8,-1) -- (-7,-1); \draw (-8,-2) -- (-7,-2); \draw (-8,-1) -- (-8,-2); \draw (-7,-1) -- (-7,-2); \draw (-8,-2.3) -- (-7,-2.3); \draw (-8,-3.3) -- (-7,-3.3); \draw (-8,-2.3) -- (-8,-3.3); \draw (-7,-2.3) -- (-7,-3.3); \draw (-8,-3.3) -- (-7,-3.3); \draw (-8,-4.3) -- (-7,-4.3); \draw (-8,-3.3) -- (-8,-4.3); \draw (-7,-3.3) -- (-7,-4.3); \node [above] at (-7.5,0) {$k_{1,2}=2$};
\draw (-2,0) -- (-1,0); \draw (-2,-1) -- (-1,-1); \draw (-2,0) -- (-2,-1); \draw (-1,0) -- (-1,-1); \draw (-2,-1) -- (-1,-1); \draw (-2,-2) -- (-1,-2); \draw (-2,-1) -- (-2,-2); \draw (-1,-1) -- (-1,-2); \draw (-2,-2.3) -- (-1,-2.3); \draw (-2,-3.3) -- (-1,-3.3); \draw (-2,-2.3) -- (-2,-3.3); \draw (-1,-2.3) -- (-1,-3.3); \draw (-2,-3.6) -- (-1,-3.6); \draw (-2,-4.6) -- (-1,-4.6); \draw (-2,-3.6) -- (-2,-4.6); \draw (-1,-3.6) -- (-1,-4.6); \node [above] at (-1.5,0) {$k_{1,2}=1,\; k_{1,1}=2$};
\draw (4,0) -- (5,0); \draw (4,-1) -- (5,-1); \draw (4,0) -- (4,-1); \draw (5,0) -- (5,-1); \draw (4,-1.3) -- (5,-1.3); \draw (4,-2.3) -- (5,-2.3); \draw (4,-1.3) -- (4,-2.3); \draw (5,-1.3) -- (5,-2.3); \draw (4,-2.6) -- (5,-2.6); \draw (4,-3.6) -- (5,-3.6); \draw (4,-2.6) -- (4,-3.6); \draw (5,-2.6) -- (5,-3.6); \draw (4,-3.9) -- (5,-3.9); \draw (4,-4.9) -- (5,-4.9); \draw (4,-3.9) -- (4,-4.9); \draw (5,-3.9) -- (5,-4.9); \node [above] at (4.5,0) {$k_{1,1}=4$};\end{tikzpicture}

\caption{\label{fig:RP4}The eleven rectangular partitions of $n=4$. The gluing
map $\pi$ takes the first rectangular partition to the Young diagram
of the partition $[4]$, the second one corresponds to $[1\,3]$,
the third and fourth ones to $[2^{2}]$, the fifth and sixth to $[1^{2}\,2]$
and the last five to $[1^{4}]$.}
\end{figure}

\begin{thm}
\label{thm:partition-formula}Given two sequences of polynomials $a_{n}(u,v),\,b_{n}(u,v)\in\mathbb{Q}[u,v]$,
satisfying: 
\begin{equation}
\sum_{n\geq0}a_{n}(u,v)t^{n}=\pexp(\sum_{n\geq1}b_{n}(u,v)t^{n})\label{eq:A-partition-B}
\end{equation}
we have: 
\[
a_{n}(u,v)=\sum_{[k]\in\mathcal{P}_{n}}\ \prod_{j=1}^{n}\frac{1}{k_{j}!}\left(\sum_{d|j}\frac{b_{d}(u^{j/d},v^{j/d})}{j/d}\right)^{k_{j}}=\sum_{[[k]]\in\mathcal{RP}_{n}}\,\prod_{l,h=1}^{n}\frac{b_{l}(u^{h},v^{h})^{k_{l,h}}}{k_{l,h}!\,h^{k_{l,h}}}.
\]
\end{thm}
\begin{rem}
As with Lemma \ref{lem:Pexp-prod}, both Lemma \ref{lem:PExp-partitions}
and the above formulae are valid for an arbitrary number $r$ of variables.
For example, we have: 
\[
a_{n}(u_{1},\cdots,u_{r})=\sum_{[[k]]\in\mathcal{RP}_{n}}\,\prod_{l,h=1}^{n}\frac{b_{l}(u_{1}^{h},\cdots,u_{r}^{h})^{k_{l,h}}}{k_{l,h}!\,h^{k_{l,h}}},
\]
when $\sum_{n\geq0}a_{n}t^{n}=\pexp(\sum_{n\geq1}b_{n}t^{n})$, and
$a_{n},b_{n}\in\mathbb{Q}[u_{1},\cdots,u_{r}]$.\end{rem}
\begin{proof}
In view of the above Remark, we consider the one variable case $x=uv$,
the general case being analogous. By setting $C_{j}(x):=\sum_{d|j}\frac{b_{d}(x^{j/d})}{j/d},$
we first show that $a_{n}(x)=\sum_{[k]\in\mathcal{P}_{n}}\prod_{j=1}^{n}\frac{C_{j}(x)^{k_{j}}}{k_{j}!}$.
This can be done by expanding the exponential:

\begin{eqnarray*}
\sum_{n\geq0}a_{n}(x)\,t^{n} & = & \exp(\Psi((b_{1}(x)t)+\Psi(b_{2}(x)t^{2})+\cdots)\\
 & = & \exp(b_{1}(x)t+\frac{1}{2}b_{1}(x^{2})t^{2}+\cdots+b_{2}(x)t^{2}+\frac{1}{2}b_{2}(x^{2})t^{4}+\cdots)\\
 & = & \exp(\sum_{n\geq1}\sum_{d|n}\frac{d}{n}b_{d}(x^{n/d})t^{n})\\
 & = & \prod_{n\geq1}\exp(C_{n}(x)t^{n})=\prod_{n\geq1}\sum_{k=0}\frac{1}{k!}C_{n}(x)^{k}t^{nk}\;.
\end{eqnarray*}
By collecting all terms contributing to a given $m=\sum jk_{j}$,
we see that we are considering partitions of $m$, in the form $[k]=[1^{k_{1}}\cdots m^{k_{m}}]$
and we get 
\[
\sum_{n\geq0}a_{n}(x)\,t^{n}=\sum_{m\geq1}(\sum_{[k]\in\mathcal{P}_{m}}\prod_{j=1}^{m}\frac{1}{k_{j}!}C_{j}(x)^{k_{j}})\,t^{m},
\]
which finishes the proof of the first expression. To prove the second
one, we need to collect all terms contributing to a given part of
size $j$: we see that we are decomposing $j=lh$, as a rectangle
of length $l$ and height $h$, where the length appears as the subscript
in the polynomials $b_{l}$, and the height appears as the power of
the variable $x$. Moreover, the coefficient of each rectangle $l\times h$
is precisely $\frac{1}{k!\,h^{k}}$ if its multiplicity is $k\geq0$.
So, we get a sum of rectangular partitions of $n$, where $[[k]]\in\mathcal{RP}_{n}$
contributes as: 
\[
\prod_{l,h=1}^{n}\frac{b_{l}(x^{h})^{k_{l,h}}}{k_{l,h}!\,h^{k_{l,h}}},
\]
as wanted. 
\end{proof}
\noindent Now, we can prove Corollary \ref{cor:main}, and write the
$E$-polynomial of each stratum (by partition type) in terms of the
irreducible lower dimensional strata. 
\begin{thm}
\label{thm:individual-strata}Let $\Gamma$ be a finitely presented
group. Then, 
\[
E(\mathcal{X}_{\Gamma}GL_{n};u,v)=\sum_{[[k]]\in\mathcal{RP}_{n}}\ \prod_{l,h=1}^{n}\frac{B_{l}^{\Gamma}(u^{h},v^{h})^{k_{l,h}}}{k_{l,h}!\,h^{k_{l,h}}},
\]
Moreover, for a given $[m]\in\mathcal{P}_{n}$, the $E$-polynomial
of the corresponding stratum is: 
\begin{eqnarray*}
E(\mathcal{X}_{\Gamma}^{[m]}GL_{n};u,v) & ={\displaystyle \sum_{[[k]]\in\pi^{-1}[m]}\ \prod_{l,h=1}^{n}\frac{B_{l}^{\Gamma}(u^{h},v^{h})^{k_{l,h}}}{k_{l,h}!\,h^{k_{l,h}}}} & ,
\end{eqnarray*}
where $B_{l}^{\Gamma}(u,v):=E(\mathcal{X}_{\Gamma}^{irr}GL_{l};u,v)$.\end{thm}
\begin{proof}
The first formula is just Theorem \ref{thm:partition-formula} for
$B_{l}^{\Gamma}(u,v)$. The second formula is immediate from the above
construction, as the only terms which contribute to a partition, i.e,
to a given Young diagram, correspond to rectangular partitions whose
image under $\pi$ is that same Young diagram.\end{proof}
\begin{example}
\label{exa:GL2} The simplest case is $n=2$, $G=GL_{2}$, where we
get that the $E$-polynomial for the character variety is given by
(dropping the superscript $\Gamma$ in $A_{i}(u,v)$ and $B_{i}(u,v)$):
\[
A_{2}(u,v)=E(\mathcal{X}_{\Gamma}GL_{2};\,u,v)=\frac{1}{2}B_{1}(u^{2},v^{2})+\frac{1}{2}B_{1}(u,v)^{2}+B_{2}(u,v),
\]
since, for each stratum, we have: 
\begin{eqnarray*}
E(\mathcal{X}_{\Gamma}^{[2]}GL_{2};\,u,v) & = & B_{2}(u,v),\\
E(\mathcal{X}_{\Gamma}^{[1^{2}]}GL_{2};\,u,v) & = & \frac{1}{2}B_{1}(u^{2},v^{2})+\frac{1}{2}B_{1}(u,v)^{2}\;.
\end{eqnarray*}

\end{example}
In the following examples, for brevity, we assume that $E(\mathcal{X}_{\Gamma}^{irr}GL_{n};\,u,v)$
only depends on the product variable $x=uv$; the 2 variable $E$-polynomial
is treated in exactly the same way. 
\begin{example}
\label{exa:GL3} Next, with $n=3$, $G=GL_{3}$, and using the same
ordering as in Figure \ref{fig:RP3}, we get 5 terms: 
\[
A_{3}(x)=E_{x}(\mathcal{X}_{\Gamma}GL_{3})=B_{3}(x)+B_{2}(x)B_{1}(x)+\frac{B_{1}(x^{3})}{3}+\frac{B_{1}(x^{2})B_{1}(x)}{2}+\frac{B_{1}(x)^{3}}{6},
\]
where the first term corresponds to $E_{x}(\mathcal{X}_{\Gamma}^{[3]}GL_{3})$,
the second to $E_{x}(\mathcal{X}_{\Gamma}^{[1\,2]}GL_{3})$, and remaining
3 terms to $E_{x}(\mathcal{X}_{\Gamma}^{[1^{3}]}GL_{3})$. 
\end{example}

\begin{example}
\label{exa:GL4} In a similar way, for $n=4$, $G=GL_{4}$, we obtain
the polynomials for each strata: 
\begin{eqnarray*}
E_{x}(\mathcal{X}_{\Gamma}^{[4]}GL_{4}) & = & B_{4}(x),\\
E_{x}(\mathcal{X}_{\Gamma}^{[1\;3]}GL_{4}) & = & B_{3}(x)B_{1}(x),\\
E_{x}(\mathcal{X}_{\Gamma}^{[2^{2}]}GL_{4}) & = & \frac{B_{2}(x)^{2}}{2}+\frac{B_{2}(x^{2})}{2},\\
E_{x}(\mathcal{X}_{\Gamma}^{[1^{2}2]}GL_{4}) & = & \frac{B_{2}(x)B_{1}(x^{2})}{2}+\frac{B_{2}(x)B_{1}(x)^{2}}{2},\\
E_{x}(\mathcal{X}_{\Gamma}^{[1^{4}]}GL_{4}) & = & \frac{B_{1}(x^{4})}{4}+\frac{B_{1}(x^{3})B_{1}(x)}{3}+\frac{B_{1}(x^{2})^{2}}{8}+\frac{B_{1}(x^{2})B_{1}(x)^{2}}{4}+\frac{B_{1}(x)^{4}}{24}\;,
\end{eqnarray*}
yielding $A_{4}(x)=E_{x}(\mathcal{X}_{\Gamma}GL_{4})$ as the sum
of these 5 strata (which comprise the $11$ terms coming from the
rectangular partitions in Figure \ref{fig:RP4}). 
\end{example}

\section{Some Explicit Computations, for low $n$}

\label{section:other_groups} In this last section, we collect several
explicit computations of $E$-polynomials of $GL_{n}$-character varieties,
and their partition type strata, for some classes of groups $\Gamma$,
including surface groups, free groups and torus knot groups. We concentrate
on the two extreme cases of these strata: the abelian stratum and
the irreducible stratum. The abelian case allows some results for
general $n$, but the $E$-polynomials of irreducible character varieties
are typically very difficult to calculate; however, for low values
of $n$, we can use some previous computations of $E$-polynomials
(obtained in most cases by point counting over finite fields, see
for example \cite{MR,BH}) to determine the $E$-polynomials and the
Euler characteristics of the irreducible character varieties, yelding
new results for the ireducible stratum. As mentioned before the irreducible
locus coincides, in many cases, with the smooth locus of the full
character variety.

We also illustrate our methods with a simple computation of the $E$-polynomial
of the Cartan brane inside the moduli space of $GL_{n}$-Higgs bundles,
an object of interest in the geometric Langlands programme (see \cite{FPN}).

\subsection{$E$-polynomials of the abelian strata}

\label{subsection:abelian}

We start by examining representations of finitely presented abelian
groups. The following result, recently obtained in \cite{FS}, deals
with the group $\Gamma=\mathbb{Z}^{r}$. 
\begin{thm}
\label{thm:-Abelian-case} Let $r,n\in\mathbb{N}$. Then $\mathcal{X}_{\mathbb{Z}^{r}}GL_{n}$
is of Hodge-Tate type and 
\[
E_{x}(\mathcal{X}_{\mathbb{Z}^{r}}GL_{n})=\sum_{[k]\in\mathcal{P}_{n}}\prod_{j=1}^{n}\frac{(x^{j}-1)^{r\,k_{j}}}{k_{j}!\ j^{k_{j}}}.
\]
\end{thm}
\begin{proof}
For the Hodge-Tate statement, see \cite{FS} or Proposition \ref{prop:reducible_part},
below. In that article, the formula is shown as a consequence of computing
the mixed Hodge-Deligne polynomial of $\mathcal{X}_{\mathbb{Z}^{r}}GL_{n}$.
For the $E$-polynomial we can also apply our main result obtaining
a simpler proof. Being an abelian group, every irreducible representation
of $\mathbb{Z}^{r}$ is one dimensional, and for $n=1$ we have 
\[
E_{x}(\mathcal{X}_{\mathbb{Z}^{r}}GL_{1})=E_{x}(\hom(\mathbb{Z}^{r},\mathbb{C}^{*}))=E_{x}((\mathbb{C}^{*})^{r})=(x-1)^{r}.
\]
Then, the generating series $\sum_{n\in\mathbb{N}}B_{n}^{\mathbb{Z}^{r}}(x)t^{n}$
of the irreducible loci reduces to 
\[
B_{1}^{\mathbb{Z}^{r}}(x)t=(x-1)^{r}t,
\]
as $B_{1}^{\mathbb{Z}^{r}}(x)=E_{x}(\mathcal{X}_{\mathbb{Z}^{r}}^{irr}GL_{1})=E_{x}(\mathcal{X}_{\mathbb{Z}^{r}}GL_{1})$,
and $B_{n}^{\mathbb{Z}^{r}}(x)\equiv0$ for all $n>1$. Then, the
generating series for $E_{x}(\mathcal{X}_{\mathbb{Z}^{r}}GL_{n})$,
by Theorem \ref{thm:A-PExpB}, is 
\[
\sum_{n\geq0}A_{n}^{\mathbb{Z}^{r}}(x)t^{n}=\pexp(B_{1}^{\mathbb{Z}^{r}}(x)t)=\sum_{n\geq0}\,\left(\sum_{[k]\in\mathcal{P}_{n}}\prod_{j=1}^{n}\frac{B_{1}^{\mathbb{Z}^{r}}(x^{j})^{k_{j}}}{k_{j}!\ j^{k_{j}}}\right)t^{n},
\]
(where the last equality comes from Lemma \ref{lem:PExp-partitions})
which immediately gives the result for $A_{n}^{\mathbb{Z}^{r}}(x)=E_{x}(\mathcal{X}_{\mathbb{Z}^{r}}GL_{n})$. 
\end{proof}
Now, let us consider a general finitely presented group $\Gamma$,
with abelianization 
\[
\Gamma_{Ab}:=\Gamma/[\Gamma,\Gamma],
\]
where $[\Gamma,\Gamma]$ is the normal subgroup generated by all commutators
in $\Gamma$ (words of the form $aba^{-1}b^{-1}$, $a,b\in\Gamma$).
It is well known that $\Gamma_{Ab}\cong\mathbb{Z}^{r}\oplus F_{N},$
where $r\in\mathbb{N}_{0}$ is called the rank of $\Gamma_{Ab}$ and
the torsion $F_{N}$ is a finite abelian group of order $N$. It is
also clear that we have 
\begin{equation}
\mathcal{R}_{\Gamma}GL_{1}=\mathcal{R}_{\Gamma}^{irr}GL_{1}\cong\mathcal{R}_{\Gamma_{Ab}}GL_{1},\label{eq:R-ab}
\end{equation}
and the same applies to the corresponding character varieties. More
generally, we have the following Lemma, that justifies calling $\mathcal{X}_{\Gamma}^{[1^{n}]}GL_{n}$
the \emph{abelian stratum}. 
\begin{lem}
\label{lem:Abelian-all1s} For every $n\in\mathbb{N}$, the abelian
stratum is isomorphic to the character variety of the abelianization
of $\Gamma$: 
\[
\mathcal{X}_{\Gamma}^{[1^{n}]}GL_{n}\cong\mathcal{X}_{\Gamma_{Ab}}GL_{n}.
\]
\end{lem}
\begin{proof}
This is a consequence of the analogous isomorphism of polystable loci:
\[
\mathcal{R}_{\Gamma}^{[1^{n}]}GL_{n}\cong\mathcal{R}_{\Gamma_{Ab}}^{ps}GL_{n},
\]
which can be shown as follows. Every polystable representation of
$\Gamma_{Ab}$ into $GL_{n}$ gives, by composition with the quotient
$\Gamma\to\Gamma_{Ab},$ a representation of $\Gamma$ which belongs
to the $[1^{n}]$ stratum, since the only irreducible representations
of an abelian group are one-dimensional. So, $\mathcal{R}_{\Gamma_{Ab}}^{ps}GL_{n}\subset\mathcal{R}_{\Gamma}^{[1^{n}]}GL_{n}$,
and the inclusion is a morphism of algebraic varieties. Conversely,
for a $[1^{n}]$-polystable representation of $\Gamma$ into $GL_{n}$,
all the generators of $\Gamma$ are sent to diagonal matrices (in
some basis, being direct sums of one-dimensional representations);
so, all commutators (the kernel of $\Gamma\to\Gamma_{Ab}$) are sent
to the identity. Thus, it defines a unique $GL_{n}$ representation
of $\Gamma_{Ab}$. \end{proof}
\begin{prop}
\label{prop:reducible_part}Let $\Gamma$ be a finitely generated
group with abelianization $\Gamma_{Ab}=\mathbb{Z}^{r}\oplus F_{N}$,
with $N=|F_{N}|$. Then, the abelian stratum is of Hodge-Tate type,
and its E-polynomial satisfies: 
\[
E_{x}(\mathcal{X}_{\Gamma}^{[1^{n}]}GL_{n})=\sum_{[k]\in\mathcal{P}_{n}}\prod_{j=1}^{n}\frac{N^{k_{j}}(x^{j}-1)^{rk_{j}}}{k_{j}!\ j^{k_{j}}}.
\]
\end{prop}
\begin{proof}
It follows from a formula of J. Cheah (\cite{Ch}), that symmetric
products of balanced varieties are balanced (see also \cite{Sil}).
Therefore, since 
\[
\mathcal{X}_{\Gamma}^{[1^{n}]}GL_{n}\cong\sym^{n}\mathcal{X}_{\Gamma}GL_{1},
\]
we only need to show that $\mathcal{X}_{\Gamma}GL_{1}$ is of Hodge-Tate
type. Clearly, 
\[
\mathcal{X}_{\Gamma}GL_{1}=\mathcal{X}_{\Gamma}^{irr}GL_{1}\cong\mathcal{R}_{\Gamma}GL_{1}\cong\mathcal{R}_{\Gamma_{Ab}}\mathbb{C}^{*}=\hom(\mathbb{Z}^{r}\oplus F_{N},\mathbb{C}^{*})
\]
Since $F_{N}$ is an abelian group of order $N\in\mathbb{N}$, it
is a direct sum of cyclic groups $\mathbb{Z}_{m}$, $m\in\mathbb{N}$.
The set $\hom(\mathbb{Z}_{m},\mathbb{C}^{*})$ is in bijection with
the $m^{th}$ roots of unity (sending the generator of $\mathbb{Z}_{m}$
to each root). Therefore, $\hom(F_{N},\mathbb{C}^{*})$ has $N$ elements
and 
\[
\hom(\mathbb{Z}^{r}\oplus F_{N},\mathbb{C}^{*})\cong\hom(\mathbb{Z}^{r},\mathbb{C}^{*})\times\hom(F_{N},\mathbb{C}^{*})\cong(\mathbb{C}^{*})^{r}\times F_{N},
\]
which is clearly of Hodge-Tate type and we get: 
\[
B_{1}^{\Gamma}(x)=E_{x}(\mathcal{X}_{\Gamma}^{irr}GL_{1})=N(x-1)^{r}.
\]
Finally, using Theorem \ref{thm:individual-strata}, the abelian stratum
$[1^{n}]$ is obtained from all rectangular partitions with a single
column (i.e. $k_{l,h}=0$ unless $l=1$, see Definition \ref{def:rectangular-partition}),
and this corresponds to usual partitions $[k]\in\mathcal{P}_{n}$.
So, we get: 
\begin{equation}
E_{x}(\mathcal{X}_{\Gamma}^{[1^{n}]}GL_{n})=\sum_{[k]\in\mathcal{P}_{n}}\prod_{j=1}^{n}\frac{B_{1}^{\Gamma}(x^{j})^{k_{j}}}{k_{j}!\ j^{k_{j}}}=\sum_{[k]\in\mathcal{P}_{n}}\prod_{j=1}^{n}\frac{N^{k_{j}}(x^{j}-1)^{rk_{j}}}{k_{j}!\ j^{k_{j}}}.\label{eq:abelian-stratum}
\end{equation}
as wanted. 
\end{proof}

\subsection{The Cartan brane of the moduli space of Higgs bundles}

\label{subsection:cartan_brane} We now illustrate the method of computation
of $E$-polynomials in a non-balanced case: the Cartan brane in the
moduli space of Higgs bundles.

The non-abelian Hodge correspondence (c.f. \cite{Sim}) establishes
a homeomorphism between $\mathcal{X}_{\Gamma_{g}}GL_{n}$, for the
surface group $\Gamma_{g}=\pi_{1}(\Sigma_{g})$ (where $\Sigma_{g}$
is a Riemann surface of genus $g$) and the moduli space $\mathcal{M}_{n}\Sigma_{g}$
of rank $n$ Higgs bundles $(E,\varphi)$ of degree zero, over $\Sigma_{g}$.
This is a singular algebraic variety, whose singular stratification
is again given in terms of partitions. More precisely, thinking of
the moduli space as parametrizing \emph{polystable Higgs bundles},
we have a disjoint union of locally closed subvarieties 
\[
\mathcal{M}_{n}\Sigma_{g}=\bigsqcup_{[k]\in\mathcal{P}_{n}}\mathcal{M}_{[k]}\Sigma_{g},
\]
where, for a partition $[k]=[1^{k_{1}}\cdots n^{k_{n}}]$, $\mathcal{M}_{[k]}\Sigma_{g}$
is the locus of Higgs bundles of the form: 
\[
\bigoplus_{j=1}^{n}(E_{j},\varphi_{j})
\]
where each $(E_{j},\varphi_{j})$ is, in turn, a direct sum of $k_{j}>0$
\emph{stable} Higgs bundles of rank $j$ (and degree zero), for $j=1,\cdots,n$
(again, if some $k_{j}=0$, the Higgs summand $(E_{j},\varphi_{j})$
is not present in the direct sum). Note that non-abelian Hodge correspondence
matches precisely the strata labelled by the same partitions. That
is, we have homeomorphisms 
\[
\mathcal{M}_{[k]}\Sigma_{g}\thickapprox\mathcal{X}_{\Gamma_{g}}^{[k]}GL_{n},
\]
since \emph{stable} Higgs bundles of a given rank $m\leq n$, correspond
to \emph{irreducible} representations of $\Gamma_{g}$ into $GL_{m}$.
However, these homeomorphisms are not holomorphic (on their smooth
loci): in fact, the Hodge structure is \emph{pure} on $\mathcal{M}_{n}\Sigma_{g}$
and \emph{mixed} (of Hodge-Tate type) on $\mathcal{X}_{\Gamma_{g}}^{[k]}GL_{n}$.

Nevertheless, we can compute the $E$-polynomial of the is the lowest
dimensional stratum, $\mathcal{M}_{[1^{n}]}\Sigma_{g}$, called the
Cartan brane in \cite{FPN}, as follows. This stratum consists of
direct sums of Higgs line-bundles of degree zero: 
\[
(L_{1},\varphi_{1})\oplus\cdots\oplus(L_{n},\varphi_{n}),
\]
and so we have: 
\[
\mathcal{M}_{[1^{n}]}\Sigma_{g}\cong\sym^{n}\mathcal{M}_{1}\Sigma_{g},
\]
where $\mathcal{M}_{1}\Sigma_{g}$ is isomorphic to the cotangent
bundle of the Jacobian of $\Sigma_{g}$, $T^{*}(J\Sigma_{g})$. Since
we are dealing with symmetric products, the analogous method for obtaining
the $E$-polynomial of $\mathcal{X}_{\Gamma_{g}}^{[1^{n}]}GL_{n}$,
as in Proposition \ref{prop:reducible_part}, gives: 
\[
E(\mathcal{M}_{[1^{n}]}\Sigma_{g};\,u,v)=\sum_{[k]\in\mathcal{P}_{n}}\prod_{j=1}^{n}\frac{B_{1}(u^{j},v^{j})^{k_{j}}}{k_{j}!\ j^{k_{j}}},
\]
where 
\[
B_{1}(u,v)=E(\mathcal{M}_{1}\Sigma_{g};\,u,v)=E(T^{*}J\Sigma_{g};\,u,v)=(uv)^{g}(1-u)^{g}(1-v)^{g},
\]
because the Jacobian $J\Sigma_{g}$ of $\Sigma_{g}$ (an abelian variety
of dimension $g$), has a pure Hodge structure and well known cohomology
ring (the cotangent bundle of $J\Sigma_{g}$ is trivial, and the term
$(uv)^{g}$ comes from the use of compactly supported cohomology).
We have thus shown the following. 
\begin{thm}
The $E$-polynomial of the Cartan brane of the moduli space of rank
$n$ Higgs bundles of degree zero is given by: 
\[
E(\mathcal{M}_{[1^{n}]}\Sigma_{g};\,u,v)=\sum_{[k]\in\mathcal{P}_{n}}\prod_{j=1}^{n}\frac{\Big((u^{j}-u^{2j})(v^{j}-v^{2j})\Big)^{k_{j}g}}{k_{j}!\ j^{k_{j}}}.
\]

\end{thm}

\subsection{$E$-polynomials of irreducible character varieties}

In \cite{MR}, Mozgovoy and Reineke obtained formulae for the $E$-polynomials
of $GL_{n}$-character varieties of the free group of rank $r$, $\Gamma=F_{r}$.
Recently, for $n=2$ and $3$, Baraglia and Hekmati derived explicit
formulae for the $E$-polynomials of $GL_{n}$-character varieties
of surface groups (both the orientable and non-orientable cases) and
for torus knot groups (as well as the cases $G=SL_{n}$, for $n=2,3$
\cite{BH}). These results were obtained by counting the number of
points of a spreading out of these character varieties, over finite
fields, and rely on a theorem of N. Katz (\cite[Appendix]{HRV1}).
Briefly, the later proves that, if there is unique polynomial that
encodes the number of points, over every finite field, of a spreading
out of a given complex variety $X$, then this polynomial agrees with
the $E$-polynomial of $X$.

In this final subsection, we use some of those formulae and our Theorem
\ref{thm:A-PExpB}, to determine $E$-polynomials of the corresponding
\emph{irreducible} character varieties, deriving $E(\mathcal{X}_{\Gamma}^{irr}GL_{n})$
from the knowledge of $E(\mathcal{X}_{\Gamma}GL_{n})$. Explicit expressions
are given in Theorems \ref{thm:E2} and \ref{thm:GL3-orient} below,
and are new results, to the best of our knowledge. As a consequence,
we obtain the numbers of irreducible components and Euler characteristics
of $\mathcal{X}_{\Gamma}^{irr}GL_{n}$.

We consider the following classes of groups $\Gamma$. If $\Sigma_{g}$
is a compact surface without boundary of genus $g\geq1$, its fundamental
group can be written as 
\[
\Gamma_{g}:=\pi_{1}(\Sigma_{g})=\left\langle a_{1},b_{1},\ldots,a_{g},b_{g}|a_{1}b_{1}a_{1}^{-1}b_{1}^{-1}\cdots a_{g}b_{g}a_{g}^{-1}b_{g}^{-1}=1\right\rangle ,
\]
and its abelianization is $(\Gamma_{g})_{Ab}=\mathbb{Z}^{2g}$, since
the unique relation is a product of commutators (belongs to $[\Gamma_{g},\Gamma_{g}]$).
A non-orientable compact surface (without boundary) of genus $k$
is a connected sum of $k$ copies of the real projective plane $\mathbb{R}\mathbb{P}^{2}$.
Its fundamental group is denoted by 
\[
\hat{\Gamma}_{k}:=\left\langle a_{1},a_{2},\ldots,a_{k}|a_{1}^{2}\cdots a_{k}^{2}=1\right\rangle ,
\]
and in this case we have: $(\hat{\Gamma}_{k})_{Ab}\cong\mathbb{Z}^{k-1}\oplus\mathbb{Z}_{2}$,
since this is the kernel of the map $\mathbb{Z}^{k}\to\mathbb{Z}$,
sending $(b_{1},\cdots,b_{k})\in\mathbb{Z}^{k}$ to $2(b_{1}+\cdots+b_{k})$,
whose vanishing corresponds to wrtiting $a_{1}^{2}\cdots a_{k}^{2}=1$
additively. As before, we let $F_{r}$ denote the free group in $r$
generators. Note that $\Gamma_{g}$, $\hat{\Gamma}_{k}$ and $F_{r}$
exhaust all fundamental groups of compact surfaces with a finite set
of points removed. Finally, consider the fundamental group 
\[
\Gamma_{a,b}=\langle x,y|x^{a}=y^{b}\rangle
\]
of the complement of a torus knot in $S^{3}$ of type $(a,b)$, where
$a,b\in\mathbb{N}$ are relatively prime. Its abelianization coincides
with the first homology group, which is $\mathbb{Z}$ (rank $1$ and
no torsion).

\subsubsection{The case of $GL_{2}$}

All the $E$-polynomials below depend on a single variable $x=uv$,
so we again use the notation $E_{x}(X):=E(X;\,\sqrt{x},\sqrt{x})$. 
\begin{thm}
\label{thm:E2} The following are the $E$-polynomials of irreducible
$GL_{2}$-character varieties for the given groups $\Gamma$: 
\begin{enumerate}
\item \label{E2-free-irr} For $\Gamma=F_{s+1}$, we have 
\[
\frac{E_{x}(\mathcal{X}_{F_{s+1}}^{irr}GL_{2})}{(x-1)^{s+1}}=(x-1)^{s}x^{s}((x+1)^{s}-1)-\frac{1}{2}(x+1)^{s}+\frac{1}{2}(x-1)^{s}.
\]

\item \label{E2-orient-irr} For $\Gamma_{g}=\pi_{1}(\Sigma_{g})$, with
$c=2g-2$, 
\[
\frac{E_{x}(\mathcal{X}_{\Gamma_{g}}^{irr}GL_{2})}{(x-1)^{c+2}}=(x^{2}-1)^{c}(x^{c}+1)+\frac{(x^{c+1}-x-1)}{2}(x+1)^{c}-\frac{(x^{c+1}-x+1)}{2}(x-1)^{c}-x^{c}.
\]

\item \label{E2-nonorient-irr} For $\hat{\Gamma}_{k}=\pi_{1}(\hat{\Sigma}_{g})$,
with $h=k-2$, 
\begin{eqnarray*}
\frac{E_{x}(\mathcal{X}_{\Gamma_{k}}^{irr}GL_{2})}{(x-1)^{h+1}} & = & 2(x^{h}+1)(x^{2}-1)^{h}+x^{h}(x-1)\frac{(x-1)^{h}+(x+1)^{h}}{2}+\\
 &  & +(2-4x^{h})(x-1)^{h}-(x+1)^{h}-2x^{h}.
\end{eqnarray*}

\item \label{E2-torus-irr} For $\Gamma_{a,b}$ we have: 
\[
\frac{E_{x}(\mathcal{X}_{\Gamma_{a,b}}^{irr}GL_{2})}{x-1}=\begin{cases}
\frac{1}{4}(a-1)(b-1)(x-2), & a,b\text{ both odd}\\
\frac{1}{4}(b-1)(ax-3a+4), & a\mbox{ even, }b\mbox{ odd}.
\end{cases}
\]

\end{enumerate}
\end{thm}
\begin{proof}
In all cases, the character varieties $\mathcal{X}_{\Gamma}GL_{2}$
were shown to be of polynomial type, so their $E$-polynomials equal
their counting polynomials computed in \cite{BH}. Thus, to obtain
the formulae above, we consider the stratification: 
\[
\mathcal{X}_{\Gamma}GL_{2}=\mathcal{X}_{\Gamma}^{[2]}GL_{2}\,\sqcup\ \mathcal{X}_{\Gamma}^{[1^{2}]}GL_{2}\cong\mathcal{X}_{\Gamma}^{irr}GL_{2}\sqcup\mathcal{X}_{\Gamma_{Ab}}GL_{2}.
\]
So, $E_{x}(\mathcal{X}_{\Gamma}^{irr}GL_{2})$ is obtained by subtracting,
from $E_{x}(\mathcal{X}_{\Gamma}GL_{2})$, the $E$-polynomial of
the abelian stratum using Proposition~\ref{prop:reducible_part},
given the rank and torsion of $\Gamma_{Ab}$.

In the free group case $\Gamma=F_{r}$, we have $\Gamma_{Ab}=\mathbb{Z}^{r}$
(no torsion), so 
\[
E_{x}(\mathcal{X}_{r}^{[1^{2}]}GL_{2})=\frac{1}{2}(x^{2}-1)^{r}+\frac{1}{2}(x-1)^{2r},
\]
and the above formula comes from \cite[Section 6.2]{BH}, where it
was shown, using $r=s+1$: 
\[
E_{x}(\mathcal{X}_{s+1}GL_{2})=(x-1)^{s+1}\left((x^{3}-x)^{s}-(x^{2}-x)^{s}+x\frac{(x+1)^{s}+(x-1)^{s}}{2}\right).
\]

For the surface group case, we have $(\Gamma_{g})_{Ab}=\mathbb{Z}^{2g}$,
and Proposition~\ref{prop:reducible_part} gives: 
\[
E_{x}(\mathcal{X}_{\Gamma_{g}}^{[1^{2}]}GL_{2})=\frac{1}{2}(x^{2}-1)^{2g}+\frac{1}{2}(x-1)^{4g}.
\]
Thus, to obtain (\ref{E2-orient-irr}), we subtract it from the formula
in \cite[Section 6.5]{BH}, which can be rewritten, letting $c=2g-2$,
as: 
\[
\frac{E_{x}(\mathcal{X}_{\Gamma_{g}}GL_{2})}{(x-1)^{c+2}}=(x^{2}-1)^{c}(x^{c}+1)+\frac{(x^{c+1}+x^{2}+x)}{2}(x+1)^{c}-\frac{(x^{c+1}-x^{2}+x)}{2}(x-1)^{c}-x^{c}.
\]
For (\ref{E2-nonorient-irr}), we have $(\hat{\Gamma}_{k})_{Ab}\cong\mathbb{Z}^{k-1}\oplus\mathbb{Z}_{2},$
so Proposition~\ref{prop:reducible_part} gives us, 
\[
E_{x}(\mathcal{X}_{\Gamma_{k}}^{[1^{2}]}GL_{2})=(x^{2}-1)^{k-1}+2(x-1)^{2k-2}=(x-1)^{k-1}[(x+1)^{k-1}+2(x-1)^{k-1}],
\]
which is subtracted from \cite[Section 6.8]{BH}, using $h=k-2$:
\begin{eqnarray*}
\frac{E_{x}(\mathcal{X}_{\Gamma_{k}}GL_{2})}{(x-1)^{h+1}} & = & 2(x^{h}+1)(x^{2}-1)^{h}+x^{h}(x-1)\frac{(x-1)^{h}+(x+1)^{h}}{2}+\\
 &  & +x[(x+1)^{h}+2(x-1)^{h}]-4(x^{2}-x)^{h}-2x^{h}.
\end{eqnarray*}
For (\ref{E2-torus-irr}), the $E$-polynomial of the character variety
of the torus knot appears in \cite[Section 6.10]{BH}: 
\[
E_{x}(\mathcal{X}_{\Gamma_{a,b}}GL_{2})=\begin{cases}
(x-1)\left(x+\frac{1}{4}(a-1)(b-1)(x-2)\right), & a,b\text{ both odd}\\
(x-1)\left(x+\frac{1}{4}(b-1)(ax-3a+4)\right), & a\mbox{ even, }b\mbox{ odd}.
\end{cases}
\]
So, to get the irreducible part, we again subtract the abelian stratum:
\[
E_{x}(\mathcal{X}_{\Gamma_{a,b}}^{[1^{2}]}GL_{2})=\frac{1}{2}(x^{2}-1)+\frac{1}{2}(x-1)^{2}=x^{2}-x,
\]
using $(\Gamma_{a,b})_{Ab}=\mathbb{Z}$ in Proposition~\ref{prop:reducible_part}. \end{proof}
\begin{cor}
The variety $\mathcal{X}_{\Gamma}^{irr}GL_{2}$ has respectively 1,
1, 2, $\frac{1}{4}(a-1)(b-1)$ and $\frac{1}{4}a(b-1)$ irreducible
components, for the groups $F_{r}$, $\Gamma_{g}$, $\hat{\Gamma}_{k}$,
$\Gamma_{a,b}$ ($ab$ odd), and $\Gamma_{a,b}$ ($a$ even, $b$
odd) respectively. The Euler characteristics of all these character
varieties are zero.\end{cor}
\begin{proof}
The number of irreducible components equals the leading coefficient
of the corresponding $E$-polynomial. For the Euler characteristic,
just substitute $x=1$ in the appropriate formulae. 
\end{proof}

\subsubsection{The case of $GL_{3}$}

In the case $n=3$, the stratification by partition type is: 
\[
\mathcal{X}_{\Gamma}GL_{3}=\mathcal{X}_{\Gamma}^{[3]}GL_{3}\sqcup\mathcal{X}_{\Gamma}^{[1\;2]}GL_{3}\sqcup\mathcal{X}_{\Gamma}^{[1^{3}]}GL_{3},
\]
and $E_{x}(\mathcal{X}_{\Gamma}^{[1\;2]}GL_{3})=E_{x}(\mathcal{X}_{\Gamma}^{irr}GL_{1})\,E_{x}(\mathcal{X}_{\Gamma}^{irr}GL_{2})$.
For the orientable surface group $\Gamma_{g}$, $E(\mathcal{X}_{\Gamma_{g}}GL_{3})$
has been computed in \cite[Section 7.1]{BH}. The abelian stratum
is also easy to get: in this case (from Example \ref{exa:GL3}), we
have 
\[
E_{x}(\mathcal{X}_{\Gamma_{g}}^{[1^{3}]}GL_{3})=\frac{B_{1}(x^{3})}{3}+\frac{B_{1}(x^{2})B_{1}(x)}{2}+\frac{B_{1}(x)^{3}}{6},
\]
with $B_{1}(x)=(x-1)^{2g}$ (as $(\Gamma_{g})_{Ab}=\mathbb{Z}^{2g}$).
Therefore, we can obtain the $E$-polynomial of the irreducible stratum
as: 
\[
E_{x}(\mathcal{X}_{\Gamma_{g}}^{irr}GL_{3})=E_{x}(\mathcal{X}_{\Gamma_{g}}GL_{3})-E_{x}(\mathcal{X}_{\Gamma_{g}}^{irr}GL_{1})\,E_{x}(\mathcal{X}_{\Gamma_{g}}^{irr}GL_{2})-E_{x}(\mathcal{X}_{\Gamma_{g}}^{[1^{3}]}GL_{3})
\]
Here, we just present the final result. 
\begin{thm}
\label{thm:GL3-orient}The irreducible $GL_{3}$-character variety
of a compact orientable surface of genus $g$ has zero Euler characteristic,
is an irreducible variety, and its $E$ polynomial, setting $c=2g-2$,
is given by: 
\begin{eqnarray*}
\frac{E_{x}(\mathcal{X}_{\Gamma_{g}}^{irr}GL_{3})}{(x-1)^{c+2}} & = & (x-1)^{2c+2}[x^{3c}-\frac{x^{c+1}}{2}-(x+1)^{c}(x^{c}+1)+\frac{1}{3}]\\
 & + & (x-1)^{2c+1}(x-2x^{2c})[\frac{x^{c}(x-2)}{2}+(x+1)^{c}(x^{c}+1)]\\
 & + & (x-1)^{2c}(x^{2}+x+1)^{c}[(x+1)^{c}(x^{3c}+1)+x^{2c}]\\
 & + & (x-1)^{2c}(x-2)x^{2c}[(x+1)^{c}(x^{c}+1)+\frac{x^{c}(x-3)}{6}]\\
 & + & \frac{(x-1)^{c+1}(x+1)^{c}}{2}[x^{c+1}-x^{3c+1}]\\
 & + & (x-1)^{c}(x^{c}-1)[x^{c-2}+x^{c+1}-2]+(x-1)^{c+2}[x^{2c-2}-x^{c-2}]\\
 & + & \frac{(x^{2}+x+1)^{c}}{3}[x^{3c+1}(x+1)-(x^{2}+x+1)]-x^{3c}\;.
\end{eqnarray*}
\end{thm}
\begin{rem}
The case of a free group of rank $r$, $\Gamma=F_{r}$, can be treated
in a completely explicit way for every $n$. In the article \cite{FNZ},
where we relate the $E$-polynomials of $\mathcal{X}_{F_{r}}G$, for
$G=GL_{n}$, $SL_{n}$ and $PGL_{n}$, we have obtained explicit formulae
for the $E$-polynomials and Euler characteristics of all partition
type strata, for all values of $r$ and $n$. \end{rem}

\end{document}